\documentclass[11pt, a4paper]{article}

 \usepackage[T1]{fontenc}
\usepackage[margin=2cm]{geometry}

\usepackage{amsmath}
\usepackage{amsthm}
\usepackage{amssymb}

\usepackage[usenames,dvipsnames]{xcolor} % load xcolor before tikz, else option clash

\usepackage{tikz-cd} 
\usetikzlibrary{cd}

\usepackage{tikz}
\usetikzlibrary{decorations.markings}
\usetikzlibrary{arrows,shapes,positioning}

\usetikzlibrary{matrix}
\usetikzlibrary{graphs}
\usetikzlibrary{backgrounds}

\usepackage{parskip}
\usepackage{bm}
\usepackage{hyperref} 
\usepackage{placeins}
\usepackage{cases}
\usepackage{float}
\usepackage{multirow}

\usepackage[matrix, arrow, curve]{xy}
\usepackage{array,longtable}   % for pagebreak in table
\newcolumntype{C}{>{$}c<{$}}  % automatic math mode, centered
\newcommand{\SO}{\mathrm{SO}}
\newcommand{\SL}{\mathrm{SL}}
\newcommand{\GL}{\mathrm{GL}}
\newcommand{\Spin}{\mathrm{Spin}}

\newcommand{\Sp}{\mathrm{Sp}}
\newcommand{\RR}{\mathbb{R}}
\newcommand{\R}{\mathbb{R}}
\newcommand{\CC}{\mathbb{C}}
\newcommand{\ZZ}{\mathbb{Z}}
\newcommand{\Z}{\mathbb{Z}}
\newcommand{\NN}{\mathbb{N}}
\newcommand{\PP}{\mathbb{P}}
\newcommand{\CA}{\mathbf{A}}
\newcommand{\Pia}{\Pi^{\mathrm{adm}}}

\newcommand{\X}{\mathcal{X}}

\newcommand\restr[2]{{% we make the whole thing an ordinary symbol
  \left.\kern-\nulldelimiterspace % automatically resize the bar with \right
  #1 % the function
  \vphantom{\big|} % pretend it's a little taller at normal size
  \right|_{#2} % this is the delimiter
  }}

\DeclareMathOperator{\cl}{cl}
\DeclareMathOperator{\Fix}{Fix}

\DeclareMathOperator{\Ad}{Ad}
\DeclareMathOperator{\Aut}{Aut}
\DeclareMathOperator{\ad}{ad}

\renewcommand{\epsilon}{\varepsilon}

\let\oldphi\phi
\renewcommand{\phi}{\varphi}

\newcommand{\iso}{\cong}

\let\hom\relax % Set  the mathoperator \hom equal to to \relax so that LaTeX thinks it's not defined
\newcommand{\hom}{\simeq}

\newcommand{\comp}{\mathsf{C}}

\newcommand\Strut{\rule{0pt}{2\normalbaselineskip}}

\newcommand{\g}{\mathfrak{g}}
\newcommand{\h}{\mathfrak{h}}
\newcommand{\calG}{G}

\title{Fundamental groups of split real Kac--Moody groups and generalized real flag manifolds}
\author{Paula Harring \hskip 1em  Ralf K{\"o}hl\\[2ex] With appendices \\ by Tobias Hartnick and Ralf K{\"o}hl and \\ by Julius Gr\"uning and Ralf K\"ohl}

\begin{document}

%
%\newtheoremstyle{stylename}% name of the style to be used
  %{spaceabove}% measure of space to leave above the theorem. E.g.: 3pt
  %{spacebelow}% measure of space to leave below the theorem. E.g.: 3pt
  %{bodyfont}% name of font to use in the body of the theorem
  %{indent}% measure of space to indent
  %{headfont}% name of head font
  %{headpunctuation}% punctuation between head and body
  %{headspace}% space after theorem head; " " = normal interword space
  %{headspec}% Manually specify head
	
	%(Any arguments that are left blank will assume their default value)

%What are the newtheoremstyle parameters for the default styles?
%The style "plain" should be equivalent to
%
%\newtheoremstyle{plain}
  %{\topsep}   % ABOVESPACE
  %{\topsep}   % BELOWSPACE
  %{\itshape}  % BODYFONT
  %{0pt}       % INDENT (empty value is the same as 0pt)
  %{\bfseries} % HEADFONT
  %{.}         % HEADPUNCT
  %{5pt plus 1pt minus 1pt} % HEADSPACE
  %{}          % CUSTOM-HEAD-SPEC
%
%The style "definition" is the same except for the body font, which is \normalfont; in "remark" the spaces above and below are 0.5\topsep, the head font is \itshape and the body font is \normalfont.

% Quelle: http://tex.stackexchange.com/questions/17551/amsthm-what-are-the-newtheoremstyle-parameters-for-the-default-styles

	% Head spec: \thmname{#1}\thmnumber{ #2}\thmnote{ (#3)}
	% 

\newtheoremstyle{defi}
  {\topsep}   % ABOVESPACE
  {\topsep}   % BELOWSPACE
  {\normalfont}  % BODYFONT
  {0pt}       % INDENT (empty value is the same as 0pt)
  {\bfseries} % HEADFONT
  {.}         % HEADPUNCT
  {5pt plus 1pt minus 1pt} % HEADSPACE
  {\thmname{#1}\thmnumber{ #2}\thmnote{ #3}}          % CUSTOM-HEAD-SPEC

	\newtheoremstyle{theorem}
  {\topsep}
  {\topsep}
  {\itshape}
  {0pt}
  {\bfseries}
  {.}
  {5pt plus 1pt minus 1pt}
  {}
	
	\newtheoremstyle{lemma}
  {\topsep}
  {\topsep}
  {\itshape}
  {0pt}
  {\bfseries}
  {.}
  {5pt plus 1pt minus 1pt}
  {\thmname{#1}\thmnumber{ #2}\thmnote{ #3}}
	
	\newtheoremstyle{namedexample}
  {\topsep}
  {\topsep}
  {\itshape}
  {0pt}
  {\bfseries}
  {.}
  {5pt plus 1pt minus 1pt}
  {\thmname{#1}:\thmnumber{ #2}\thmnote{ #3}}

{\theoremstyle{theorem}
\newtheorem{thm}{Theorem}[section]
\newtheorem*{thm*}{Theorem}
}

{\theoremstyle{lemma}
\newtheorem{lem}[thm]{Lemma}
\newtheorem*{lem*}{Lemma}
\newtheorem{prop}[thm]{Proposition}
\newtheorem{cor}[thm]{Corollary}
}

{\theoremstyle{defi}
\newtheorem{rem}[thm]{Remark}
\newtheorem{nota}[thm]{Notation}
\newtheorem{de+rem}[thm]{Definition and Remark}
\newtheorem{de}[thm]{Definition}
\newtheorem{ex}[thm]{Example}
\newtheorem{observation}[thm]{Observation}
}

{\theoremstyle{namedexample}
\newtheorem*{ex*}{Example}
}

\maketitle

%
%\abstract{We determine the isomorphism types of the fundamental groups of $G(\Pi)$, its maximal compact subgroup $K(\Pi)$, and its spin covers for all algebraically simply connected semisimple split real Kac--Moody groups $G(\Pi)$ with the property that the Bruhat decompositions of the generalized flag varieties are CW decompositions -- in particular, for all two-spherical groups.}

\abstract{We determine the fundamental groups of symmetrizable algebraically simply connected split real Kac--Moody groups endowed with the Kac--Peterson topology. In analogy to the finite-dimensional situation, because of the Iwasawa decomposition $G = KAU^+$ the embedding $K \hookrightarrow G$ is a weak homotopy equivalence, in particular $\pi_1(G) = \pi_1(K)$. It thus suffices to determine $\pi_1(K)$ which we achieve by investigating the fundamental groups of generalized flag varieties. Our results apply in all cases in which the Bruhat decomposition of the generalized flag variety is a CW decomposition -- in particular, we cover the complete symmetrizable situation; furthermore, the results concerning only the structure of $\pi_1(K)$ actually also hold in the non-symmetrizable two-spherical case.}

%
%whose Bruhat decomposition
%
%\abstract{We prove that the spin covers of the maximal compact subgroups of irreducible simply laced split real Kac--Moody groups are topologically simply connected with respect to the Kac--Peterson topology.}\\

%
%\lila{Lila: Geaendert}\\
%\rot{Rot: Muss noch bearbeitet werden.}\\
%\gruen{Gruen: Kann wahrscheinlich weg.}\\
%\blau{Blau: Nicht unbedingt notwendige Zusatzinformationen}\\

\section{Introduction} \label{setting}

The structure of maximal compact subgroups in semisimple Lie groups was investigated by Cartan and, later, Mostow: In \cite{Mos}, Mostow gives a new proof of a Cartan's theorem stating that a connected semisimple Lie group $G$ is a topological product of a maximal compact subgroup $K$ and a Euclidean space, implying in particular that $G$ and $K$ have isomorphic fundamental groups. %\rot{Mention that Cartan and Mostow determined isomorphism types of maximal compact subgroups (Reference?)}.
Subsequent case-by-case analysis provided the isomorphism types of these maximal compact subgroups -- which in the split real situation turn out to be all classical -- and their fundamental groups; tables of the maximal compact subgroups can be found in \cite[p 518]{Helga}, their fundamental groups in \cite[94.33]{Salz}.

Starting in the 1940's, Dynkin diagrams, introduced in \cite{Dynkin}, have been used to describe the structure of simple Lie groups.
In this article, we present a uniform result which makes it possible to determine the fundamental group of any algebraically simply connected split real simple Lie group -- and, more generally, any algebraically simply-connected semisimple split real topological Kac--Moody group -- directly from its Dynkin diagram.

In \cite[Theorem 1]{Ti}, Tits for every generalized Cartan matrix $\CA$ provides a functor $\calG_\CA : \mathfrak{cRings} \to \mathfrak{Groups}$ from commutative rings into groups. Let $\Pi$ be the Dynkin diagram of $\CA$. 
\begin{de}
\label{Kac--Moody group}
 We set $G(\Pi) := [\calG_\CA(\RR),\calG_\CA(\RR)]$ and refer to this group as the \emph{algebraically simply-connected semisimple split real Kac--Moody group of type $\Pi$}.
\end{de}

Kac--Moody groups endowed with the Kac--Peterson topology have been studied extensively by the second author together with Gl\"ockner and Hartnick in \cite{GGH} and with Hartnick and Mars in \cite{HKM}. Our result is applicable to those Kac--Moody groups whose Bruhat decompositions are CW decompositions and for which the embedding $K \hookrightarrow G$ is a weak homotopy equivalence.

\medskip
In order to fix notations, 
let $G=G(\Pi)$ be the algebraically simply-connected split real semisimple Kac--Moody group associated to an irreducible diagram $\Pi = (V,E)$ endowed with the Kac--Peterson topology (for definitions, see Section \ref{Kac--Moody-basics}). % whose vertices are labelled by some finite set $I$.
Let $K=K(\Pi)$ be the so-called maximal compact subgroup of the topological group $G(\Pi)$, i.e., the subgroup fixed by the Cartan--Chevalley involution $\theta$ of $G(\Pi)$. We stress that in the infinite-dimensional non-Lie case this topological group $K$ is {\em not} a compact group, only a $k$-group, in fact a $k_\omega$-group.

Given the Dynkin diagram $\Pi = (V,E)$ with a fixed labelling $\lambda:\{1, \dots, n\} =: I \to  V $, we define a modified diagram $\Pia$ with vertex set $V$ and $\{i^\lambda,j^\lambda\} \in V \times V$ edge if and only if $\epsilon(i,j) = \epsilon(j,i) = -1$, where $\epsilon(i,j)$ denotes the parity of the corresponding Cartan matrix entry. To each connected component $\bar{\Pi}^{\mathrm{adm}}$ of $\Pia$ we then assign a colour as follows: Let $\bar{\Pi}^{\mathrm{adm}}$ be coloured red (denoted by $r$) if it contains a vertex $i^\lambda$ such that there exists a vertex $j^\lambda \in V$ satisfying $\epsilon(i,j) = 1$ and $ \epsilon(j,i) = -1$; let $\bar{\Pi}^{\mathrm{adm}}$ be coloured green ($g$) if it is not red and consists only of an isolated vertex; and blue ($b$) else. 

One can then read off the isomorphism type of $\pi_1(G(\Pi))$ from the coloured diagram $\Pia$ as specified in the following theorem.     

\begin{thm*}
Let $\Pi$ be an irreducible Dynkin diagram such that the Bruhat decomposition of $G(\Pi)$ provides a CW decomposition (i.e., such that the conclusion of Proposition \ref{bruhat decomp is cw decomp} holds) and such that the embeding $K \hookrightarrow G(\Pi)$ is a weak homotopy equivalence (i.e., such that the conclusion of Theorem~\ref{FundamentalGroups2} holds). Let $n(g)$ and $n(b)$ be the number of connected components of $\Pia$ of colour $g$ and $b$, respectively. Then 
\[\pi_1(G(\Pi)) \iso \ZZ^{n(g)} \times C_2^{n(b)}.\] In particular, this statement holds in the symmetrizable case.
\end{thm*}  

\begin{ex*}[Isomorphism types of $\pi_1(G(\Pi))$ for the spherical Dynkin diagrams\footnote{Dynkin diagram LaTeX styles kindly provided by Max Horn at \cite{Ho}}]

\def\DynkinNodeSize{3.5mm}
\def\DynkinArrowLength{3mm}
\tikzset{
  % a diagram node
  dnode/.style={
    circle,
    inner sep=0pt,
    minimum size=\DynkinNodeSize,
    fill=white,
    draw},
  %
	%coloured nodes
	  gnode/.style={
    circle,
    inner sep=0pt,
    minimum size=\DynkinNodeSize,
    fill=green,
    draw},
		bnode/.style={
    circle,
    inner sep=0pt,
    minimum size=\DynkinNodeSize,
    fill=blue,
    draw},
		rnode/.style={
    circle,
    inner sep=0pt,
    minimum size=\DynkinNodeSize,
    fill=red,
    draw},
  middlearrow/.style={
    decoration={markings,
      mark=at position 0.6 with
      %{\arrow[black]{angle 90};}
      %{\arrow[black]{angle 60};}
      %{\arrow[black]{stealth};}
      {\draw (0:0mm) -- +(+135:\DynkinArrowLength); \draw (0:0mm) -- +(-135:\DynkinArrowLength);},
    },
    postaction={decorate}
  },
  leftrightarrow/.style={
    decoration={markings,
      mark=at position 0.999 with
      {
      \draw (0:0mm) -- +(+135:\DynkinArrowLength); \draw (0:0mm) -- +(-135:\DynkinArrowLength);
      },
      mark=at position 0.001 with
      {
      \draw (0:0mm) -- +(+45:\DynkinArrowLength); \draw (0:0mm) -- +(-45:\DynkinArrowLength);
      },
    },
    postaction={decorate}
  },
  % single edge
  sedge/.style={
  },
  % directed double edge
  dedge/.style={
    middlearrow,
    double distance=0.5mm,
  },
  % directed triple edge
  tedge/.style={
    middlearrow,
    double distance=1.0mm+\pgflinewidth,
    postaction={draw}, % third line
  },
  % double edge with two arrows, for \tilde{A}_1 residues
  infedge/.style={
    leftrightarrow,
    double distance=0.5mm,
  },
}
\[
\begin{array}{|c|c|c|}
\hline \rule{0pt}{\normalbaselineskip}
\Pi & \Pia \;\mathrm{coloured\;by}\;\gamma & \pi_1(G(\Pi))\\ 
\hline \Strut
\begin{tikzpicture}%[every node/.style={dnode}]
    \draw (-0.5,0) node[anchor=east]  {$A_1$};

    \node[dnode] (1) at (0,0) {};
    %\node[dnode] (2) at (1,0) {};
    %\node[dnode] (3) at (3,0) {};
    %\node[dnode] (4) at (4,0) {};
    %\node[dnode,label=below:$n$] (5) at (5,0) {};

    %\path (1) edge[sedge] (2)
          %(2) edge[sedge,dashed] (3)
          %(3) edge[sedge] (4)
          %(4) edge[sedge] (5)
          ;
\end{tikzpicture}
&  
\begin{tikzpicture}%[every node/.style={dnode}]
    %\draw (-1,0) node[anchor=east]  {$A_n$};
    %\node[bnode,label=center:$b$] (1) at (0,0) {};
    %\node[bnode,label=center:$b$] (2) at (1,0) {};
    %\node[bnode,label=center:$b$] (3) at (3,0) {};
    %\node[bnode,label=center:$b$] (4) at (4,0) {};
    \node[gnode,label=center:$g$] (1) at (0,0) {};
    %\node[bnode,label=center:$b$] (2) at (1,0) {};
    %\node[bnode,label=center:$b$] (3) at (3,0) {};
    %\node[bnode,label=center:$b$] (4) at (4,0) {};
    %\node[dnode,label=below:$n$] (5) at (5,0) {};

    %\path (1) edge[sedge] (2)
          %(2) edge[sedge,dashed] (3)
          %(3) edge[sedge] (4)
          %(4) edge[sedge] (5)
          ;
\end{tikzpicture}&
\pi_1(\SL_{2}(\RR)) \iso  \ZZ   \Strut \\ \hline \Strut
\begin{tikzpicture}%[every node/.style={dnode}]
    \draw (-0.5,0) node[anchor=east]  {$A_n$};

    \node[dnode] (1) at (0,0) {};
    \node[dnode] (2) at (1,0) {};
    \node[dnode] (3) at (3,0) {};
    \node[dnode] (4) at (4,0) {};
    %\node[dnode,label=below:$n$] (5) at (5,0) {};

    \path (1) edge[sedge] (2)
          (2) edge[sedge,dashed] (3)
          (3) edge[sedge] (4)
          %(4) edge[sedge] (5)
          ;
\end{tikzpicture}
&  
\begin{tikzpicture}%[every node/.style={dnode}]
    %\draw (-1,0) node[anchor=east]  {$A_n$};
    %\node[bnode,label=center:$b$] (1) at (0,0) {};
    %\node[bnode,label=center:$b$] (2) at (1,0) {};
    %\node[bnode,label=center:$b$] (3) at (3,0) {};
    %\node[bnode,label=center:$b$] (4) at (4,0) {};
    \node[bnode,label=center:$b$] (1) at (0,0) {};
    \node[bnode,label=center:$b$] (2) at (1,0) {};
    \node[bnode,label=center:$b$] (3) at (3,0) {};
    \node[bnode,label=center:$b$] (4) at (4,0) {};
    %\node[dnode,label=below:$n$] (5) at (5,0) {};

    \path (1) edge[sedge] (2)
          (2) edge[sedge,dashed] (3)
          (3) edge[sedge] (4)
          %(4) edge[sedge] (5)
          ;
\end{tikzpicture}&
\pi_1(\SL_{n+1}(\RR)) \iso C_2 \quad (n\geq 2) \Strut \\ \hline \Strut
\begin{tikzpicture}%[every node/.style={dnode}]
    \draw (-0.5,0) node[anchor=east]  {$B_n$};

    \node[dnode] (1) at (0,0) {};
    \node[dnode] (2) at (1,0) {};
    \node[dnode] (3) at (3,0) {};
    \node[dnode] (4) at (4,0) {};
    %\node[dnode,label=below:$n$] (5) at (5,0) {};

    \path (1) edge[sedge] (2)
          (2) edge[sedge,dashed] (3)
          (3) edge[dedge] (4)
          %(4) edge[dedge] (5)
          ;
\end{tikzpicture}&
\begin{tikzpicture}%[every node/.style={dnode}]
    %\draw (-0.5,0) node[anchor=east]  {$B_n$};

    \node[bnode,label=center:$b$] (1) at (0,0) {};
    \node[bnode,label=center:$b$] (2) at (1,0) {};
    \node[bnode,label=center:$b$] (3) at (3,0) {};
    \node[rnode,label=center:$r$] (4) at (4,0) {};
    %\node[dnode,label=below:$n$] (5) at (5,0) {};

    \path (1) edge[sedge] (2)
          (2) edge[sedge,dashed] (3)
          %(3) edge[sedge] (4)
          %(4) edge[dedge] (5)
          ;
\end{tikzpicture}& \pi_1(\Spin(n,n+1)) \iso \begin{cases} \ZZ &\mathrm{if }\;n\leq 2,\\C_2&\mathrm{if }\; n>2.\end{cases}\Strut \\ \hline \Strut
\begin{tikzpicture}%[every node/.style={dnode}]
    \draw (-0.5,0) node[anchor=east]  {$C_n$};

    \node[dnode] (1) at (0,0) {};
    \node[dnode] (2) at (1,0) {};
    \node[dnode] (3) at (3,0) {};
    \node[dnode] (4) at (4,0) {};
    %\node[dnode,label=below:$n$] (5) at (5,0) {};

    \path (1) edge[sedge] (2)
          (2) edge[sedge,dashed] (3)
          (4) edge[dedge] (3)
          %(5) edge[dedge] (4)
          ;
\end{tikzpicture}
&
\begin{tikzpicture}%[every node/.style={dnode}]
    %\draw (-1,0) node[anchor=east]  {$C_n$};

    \node[rnode,label=center:$r$] (1) at (0,0) {};
    \node[rnode,label=center:$r$] (2) at (1,0) {};
    \node[rnode,label=center:$r$] (3) at (3,0) {};
    \node[gnode,label=center:$g$] (4) at (4,0) {};
    %\node[dnode,label=below:$n$] (5) at (5,0) {};

    \path (1) edge[sedge] (2)
          (2) edge[sedge,dashed] (3)
          %(3) edge[sedge] (4)
          %(5) edge[dedge] (4)
          ;
\end{tikzpicture}& \pi_1(\Sp(2n,\RR)) \iso \ZZ\Strut \\ \hline \Strut % hier muss eine hline und danach Abstand hin
\begin{tikzpicture}%[every node/.style={dnode}]
    \draw (-0.5,0) node[anchor=east]  {$D_n$};

    \node[dnode] (1) at (0,0) {};
    \node[dnode] (2) at (1,0) {};
    \node[dnode] (3) at (3,0) {};
    \node[dnode] (4) at (4,0) {};
		\node[dnode] (5) at (3,1) {};
    %\node[dnode,label=below:$n$] (5) at (5,0) {};

    \path (1) edge[sedge] (2)
          (2) edge[sedge,dashed] (3)
          (3) edge[sedge] (4)
					(3) edge[sedge] (5)
          %(4) edge[dedge] (5)
          ;
\end{tikzpicture}&
\begin{tikzpicture}%[every node/.style={dnode}]
    %\draw (-0.5,0) node[anchor=east]  {$B_n$};

    \node[bnode,label=center:$b$] (1) at (0,0) {};
    \node[bnode,label=center:$b$] (2) at (1,0) {};
    \node[bnode,label=center:$b$] (3) at (3,0) {};
    \node[bnode,label=center:$b$] (4) at (4,0) {};
		\node[bnode,label=center:$b$] (5) at (3,1) {};
    %\node[dnode,label=below:$n$] (5) at (5,0) {};

    \path (1) edge[sedge] (2)
          (2) edge[sedge,dashed] (3)
          (3) edge[sedge] (4)
          (3) edge[sedge] (5)
          ;
\end{tikzpicture}& \pi_1(\Spin(n,n)) \iso C_2 \quad (n\geq 3) \Strut \\ \hline \Strut
\begin{tikzpicture}%[every node/.style={dnode}]
    \draw (-0.5,0) node[anchor=east]  {$F_4$};

    \node[dnode] (1) at (0,0) {};
    \node[dnode] (2) at (1,0) {};
		\node[dnode] (3) at (2,0) {};
    \node[dnode] (4) at (3,0) {};
    %\node[dnode,label=below:$n$] (5) at (5,0) {};

    \path (1) edge[sedge] (2)
		 (3) edge[dedge] (2)
		 (3) edge[sedge] (4)
          ;
\end{tikzpicture}
&  
\begin{tikzpicture}%[every node/.style={dnode}]
    \draw (-0.5,0) node[anchor=east]  {$F_4$};

    \node[rnode,label=center:$r$] (1) at (0,0) {};
    \node[rnode,label=center:$r$] (2) at (1,0) {};
		\node[bnode,label=center:$b$] (3) at (2,0) {};
    \node[bnode,label=center:$b$] (4) at (3,0) {};
    %\node[dnode,label=below:$n$] (5) at (5,0) {};

    \path (1) edge[sedge] (2)
		 %(3) edge[sedge] (2)
		 (3) edge[sedge] (4)
          ;
\end{tikzpicture}&
\pi_1(F_4) \iso C_2 \Strut\\ \hline \Strut

\begin{tikzpicture}%[every node/.style={dnode}]
    \draw (-0.5,0) node[anchor=east]  {$G_2$};

    \node[dnode] (1) at (0,0) {};
    \node[dnode] (2) at (1,0) {};
    %\node[dnode,label=below:$n$] (5) at (5,0) {};

    \path (1) edge[tedge] (2)
          ;
\end{tikzpicture}
&  
\begin{tikzpicture}%[every node/.style={dnode}]
    %\draw (-1,0) node[anchor=east]  {$A_n$};
    %\node[bnode,label=center:$b$] (1) at (0,0) {};
    %\node[bnode,label=center:$b$] (2) at (1,0) {};
    %\node[bnode,label=center:$b$] (3) at (3,0) {};
    %\node[bnode,label=center:$b$] (4) at (4,0) {};
    \node[bnode,label=center:$b$] (1) at (0,0) {};
    \node[bnode,label=center:$b$] (2) at (1,0) {};

    \path (1) edge[sedge] (2)
          ;
\end{tikzpicture}&
\pi_1(G_{2,2}) \iso C_2  \\ \hline \end{array} \]
%\Strut \\ \hline \Strut
\end{ex*}

\begin{ex*}[Isomorphism types of $\pi_1(G(\Pi))$ for selected indefinite Dynkin diagrams\footnote{Dynkin diagram LaTeX styles kindly provided by Max Horn at \cite{Ho}}]

\def\DynkinNodeSize{3.5mm}
\def\DynkinArrowLength{3mm}
\tikzset{
  % a diagram node
  dnode/.style={
    circle,
    inner sep=0pt,
    minimum size=\DynkinNodeSize,
    fill=white,
    draw},
  %
	%coloured nodes
	  gnode/.style={
    circle,
    inner sep=0pt,
    minimum size=\DynkinNodeSize,
    fill=green,
    draw},
		bnode/.style={
    circle,
    inner sep=0pt,
    minimum size=\DynkinNodeSize,
    fill=blue,
    draw},
		rnode/.style={
    circle,
    inner sep=0pt,
    minimum size=\DynkinNodeSize,
    fill=red,
    draw},
  middlearrow/.style={
    decoration={markings,
      mark=at position 0.6 with
      %{\arrow[black]{angle 90};}
      %{\arrow[black]{angle 60};}
      %{\arrow[black]{stealth};}
      {\draw (0:0mm) -- +(+135:\DynkinArrowLength); \draw (0:0mm) -- +(-135:\DynkinArrowLength);},
    },
    postaction={decorate}
  },
  leftrightarrow/.style={
    decoration={markings,
      mark=at position 0.999 with
      {
      \draw (0:0mm) -- +(+135:\DynkinArrowLength); \draw (0:0mm) -- +(-135:\DynkinArrowLength);
      },
      mark=at position 0.001 with
      {
      \draw (0:0mm) -- +(+45:\DynkinArrowLength); \draw (0:0mm) -- +(-45:\DynkinArrowLength);
      },
    },
    postaction={decorate}
  },
  % single edge
  sedge/.style={
  },
  % directed double edge
  dedge/.style={
    middlearrow,
    double distance=0.5mm,
  },
  % directed triple edge
  tedge/.style={
    middlearrow,
    double distance=1.0mm+\pgflinewidth,
    postaction={draw}, % third line
  },
  % double edge with two arrows, for \tilde{A}_1 residues
  infedge/.style={
    leftrightarrow,
    double distance=0.5mm,
  },
}

\[
\begin{array}{|c|c|c|}
\hline \rule{0pt}{\normalbaselineskip}
\Pi & \Pia \;\mathrm{coloured\;by}\;\gamma & \pi_1(G(\Pi))\\ 
\hline \Strut
%\begin{tikzpicture}%[every node/.style={dnode}]
    %\draw (-0.5,0) node[anchor=east]  {$F_4$};
%
    %\node[dnode] (1) at (0,0) {};
    %\node[dnode] (2) at (1,0) {};
		%\node[dnode] (3) at (2,0) {};
    %\node[dnode] (4) at (3,0) {};
    %%\node[dnode,label=below:$n$] (5) at (5,0) {};
%
    %\path (1) edge[sedge] (2)
		 %(3) edge[dedge] (2)
		 %(3) edge[sedge] (4)
          %;
%\end{tikzpicture}
%&  
%\begin{tikzpicture}%[every node/.style={dnode}]
    %\draw (-0.5,0) node[anchor=east]  {$F_4$};
%
    %\node[rnode,label=center:$r$] (1) at (0,0) {};
    %\node[rnode,label=center:$r$] (2) at (1,0) {};
		%\node[bnode,label=center:$b$] (3) at (2,0) {};
    %\node[bnode,label=center:$b$] (4) at (3,0) {};
    %%\node[dnode,label=below:$n$] (5) at (5,0) {};
%
    %\path (1) edge[sedge] (2)
		 %%(3) edge[sedge] (2)
		 %(3) edge[sedge] (4)
          %;
%\end{tikzpicture}&
%\pi_1(F_4) \iso C_2 \Strut\\ \hline \Strut

\begin{tikzpicture}%[every node/.style={dnode}]
    \draw (-0.5,-3.5) node[anchor=east]  {$E_{10}$};
		 \node[dnode] (1) at (0,0) {};
    \node[dnode] (2) at (0,-1) {};
    \node[dnode] (3) at (0,-2) {};
    \node[dnode] (4) at (0,-3) {};
		\node[dnode] (5) at (0,-4) {};
		\node[dnode] (6) at (0,-5) {};
		\node[dnode] (7) at (0,-6) {};
		\node[dnode] (8) at (0,-7) {};
		\node[dnode] (9) at (0,-8) {};
		\node[dnode] (10) at (1,-6) {};

    %\node[dnode,label=below:$n$] (5) at (5,0) {};

    \path (1) edge[sedge] (2)
          (2) edge[sedge] (3)
          (3) edge[sedge] (4)
          (4) edge[sedge] (5)
					(5) edge[sedge] (6)
					(6) edge[sedge] (7)
					(7) edge[sedge] (8)
					(8) edge[sedge] (9)
					(7) edge[sedge] (10)
          ;
\end{tikzpicture}
&
\begin{tikzpicture}%[every node/.style={dnode}]
    %\draw (-0.5,-3.5) node[anchor=east]  {$E_{10}$};

    \node[bnode,label=center:$b$] (1) at (0,0) {};
    \node[bnode,label=center:$b$] (2) at (0,-1) {};
    \node[bnode,label=center:$b$] (3) at (0,-2) {};
    \node[bnode,label=center:$b$] (4) at (0,-3) {};
		\node[bnode,label=center:$b$] (5) at (0,-4) {};
		\node[bnode,label=center:$b$] (6) at (0,-5) {};
		\node[bnode,label=center:$b$] (7) at (0,-6) {};
		\node[bnode,label=center:$b$] (8) at (0,-7) {};
		\node[bnode,label=center:$b$] (9) at (0,-8) {};
		\node[bnode,label=center:$b$] (10) at (1,-6) {};
    %\node[dnode,label=below:$n$] (5) at (5,0) {};

    \path (1) edge[sedge] (2)
          (2) edge[sedge] (3)
          (3) edge[sedge] (4)
          (4) edge[sedge] (5)
					(5) edge[sedge] (6)
					(6) edge[sedge] (7)
					(7) edge[sedge] (8)
					(8) edge[sedge] (9)
					(7) edge[sedge] (10)
          ;
\end{tikzpicture}& \pi_1(E_{10}) \iso C_2 \Strut \\ \hline \Strut

\begin{tikzpicture}%[every node/.style={dnode}]
    \draw (-0.5,-2.5) node[anchor=east]  {$X$};
%\node at (0,4,5){};
    \node[dnode] (1a) at (0,0) {};
    \node[dnode] (1b) at (1,0) {};
    \node[dnode] (1c) at (2,0) {};
    \node[dnode] (2a) at (0,-1) {};
    \node[dnode] (2b) at (1,-1) {};
    \node[dnode] (2c) at (2,-1) {};
	  \node[dnode] (3a) at (0,-2) {};
    \node[dnode] (3b) at (1,-2) {};
    \node[dnode] (3c) at (2,-2) {};
	  \node[dnode] (4a) at (0,-3) {};
    \node[dnode] (4b) at (1,-3) {};
    \node[dnode] (4c) at (2,-3) {};
	  \node[dnode] (5a) at (0,-4) {};
    \node[dnode] (5b) at (1,-4) {};
    \node[dnode] (5c) at (2,-4) {};	
	  \node[dnode] (6a) at (0,-5) {};
    %\node[dnode,label=below:$n$] (5) at (5,0) {};

    \path (1a) edge[tedge] (1b)
          (1b) edge[sedge] (1c)
					(2a) edge[dedge] (1a)
					(2a) edge[dedge] (2b)
					(2c) edge[dedge] (1c)
					(2c) edge[dedge] (2b)
					(2c) edge[dedge] (3c)
					(3a) edge[sedge] (2a)
					(3b) edge[dedge] (3a)
					(3b) edge[dedge] (3c)
					(3b) edge[dedge] (4b)
					(3c) edge[tedge] (4c)
					(4a) edge[dedge] (3a)
					(4a) edge[dedge] (4b)
					(4a) edge[sedge] (5a)
					(5b) edge[dedge] (4b)
					(5b) edge[sedge] (5c)
					(5a) edge[sedge] (6a)
          ;
\end{tikzpicture}
&
\begin{tikzpicture}%[every node/.style={dnode}]
    %\draw (-0.5,0) node[anchor=east]  {$C_n$};
		\node at (0,.5){};
    \node[rnode,label=center:$r$] (1a) at (0,0) {};
    \node[rnode,label=center:$r$] (1b) at (1,0) {};
    \node[rnode,label=center:$r$] (1c) at (2,0) {};
    \node[rnode,label=center:$r$] (2a) at (0,-1) {};
    \node[rnode,label=center:$r$] (2b) at (1,-1) {};
    \node[gnode,label=center:$g$] (2c) at (2,-1) {};
	  \node[rnode,label=center:$r$] (3a) at (0,-2) {};
    \node[gnode,label=center:$g$] (3b) at (1,-2) {};
    \node[rnode,label=center:$r$] (3c) at (2,-2) {};
	  \node[bnode,label=center:$b$] (4a) at (0,-3) {};
    \node[rnode,label=center:$r$] (4b) at (1,-3) {};
    \node[rnode,label=center:$r$] (4c) at (2,-3) {};
	  \node[bnode,label=center:$b$] (5a) at (0,-4) {};
    \node[bnode,label=center:$b$] (5b) at (1,-4) {};
    \node[bnode,label=center:$b$] (5c) at (2,-4) {};	
	  \node[bnode,label=center:$b$] (6a) at (0,-5) {};
    %\node[dnode,label=below:$n$] (5) at (5,0) {};

    \path (1a) edge[sedge] (1b)
          (1b) edge[sedge] (1c)
					%(2a) edge[dedge] (1a)
					%(2a) edge[dedge] (2b)
					%(2c) edge[dedge] (1c)
					%(2c) edge[dedge] (2b)
					%(2c) edge[dedge] (3c)
					(3a) edge[sedge] (2a)
					%(3b) edge[dedge] (3a)
					%(3b) edge[dedge] (3c)
					%(3b) edge[dedge] (4b)
					(3c) edge[sedge] (4c)
					%(4a) edge[dedge] (3a)
					%(4a) edge[dedge] (4b)
					(4a) edge[sedge] (5a)
					%(5b) edge[dedge] (4b)
					(5b) edge[sedge] (5c)
					(5a) edge[sedge] (6a)
          ; 
\end{tikzpicture}& \pi_1(G(X)) \iso \ZZ^2 \times C_2^2 \\ \hline 
\end{array}
\]
%\rot{Die Gruppen in der rechten Spalte sollten nicht am Boden der Tabellenzelle kleben, sondern mittig positioniert werden. Ist mir bisher nicht gelungen.}

%
%\[\scalebox{2}{$\begin{array}{ccc}
	%\dynkin{B}{5}& \rightsquigarrow &\dynkin{B}{5}
%\end{array}$}\]

\end{ex*}

While in the classical finite-dimensional Lie case, one has a topological Iwasawa decomposition $G = K\times A \times U^+$ with $A$ and $U^+$ contractible, implying $\pi_1(K) \iso \pi_1(G)$, it is currently unknown whether the corresponding Iwasawa decomposition in the general Kac--Moody case is also topological. However, using a fibration result by Palais (see Proposition~\ref{Palais}), in the appendix Hartnick and the second author prove  that the isomorphism between the fundamental groups still exists in the general symmetrizable case, therefore reducing the problem to the computation of $\pi_1(K)$.

In \cite[Section~16]{GHKW}, the group ${\Spin(\Pi,\kappa)}$ -- where $\kappa$ denotes a so-called \emph{admissible colouring} of the vertices of $\Pi$ -- is defined as the canonical universal enveloping group of a $\Spin(2)$-amalgam $\mathcal{A}(\Pi, \Spin(2)) = \{\tilde{G}_{ij}, \tilde{\oldphi}_{ij}^i \mid i \neq j \in I\}$ where the isomorphism type of $\tilde{G}_{ij}$ depends on the $(i,j)$- and $(j,i)$-entries of the Cartan matrix of $\Pi$ as well as the values of $\kappa$ on the corresponding vertices.

It is shown in \cite[Section~17]{GHKW} that there exists a {\em finite} central extension $\Spin(\Pi, \kappa) \to K(\Pi)$ which implies that the subspace topology on $K(\Pi)$ inherited from the Kac--Peterson topology on $G(\Pi)$ defines a unique topology on $\Spin(\Pi, \kappa)$ that turns the central extension into a covering map. 
The resulting group topology on $\Spin(\Pi, \kappa)$ is called the {\em Kac--Peterson topology} on $\Spin(\Pi, \kappa)$. 

In the simply-laced case, there is a unique non-trivial admissible colouring $\kappa$ and the corresponding group  $\Spin(\Pi):= \Spin(\Pi,\kappa)$ double-covers $K$ as shown in \cite{GHKW}. We prove here that in the simply-laced case $\Spin(\Pi)$ is simply connected which then implies that $\pi_1(K) \iso C_2$.  

The strategy of proof in the simply-laced case is to study fibre bundles of the form $$\Spin(3) \to \Spin(\Pi) \to \Spin(\Pi)/\Spin(3)$$ arising from embeddings of $\Spin(3)$ along subdiagrams of type $A_2$, which yield exact sequences of the form \[ \{1\} = \pi_1(\Spin(3)) \rightarrow \pi_1(\Spin(\Pi)) \rightarrow \pi_1(\Spin(\Pi)/\Spin(3))\]
and establishes the equivalence of simple-connectedness of $\Spin(\Pi)$ with the simple-connectedness of $\Spin(\Pi)/\Spin(3)$. 

A key to the proof both in the simply-laced and in the general case is the computation of the fundamental groups of \emph{generalized flag varieties} -- that is, spaces of the form $G/P_J$ for a parabolic subgroup $P_J$ of $G$ corresponding to an index subset $J \subseteq I$. It turns out that the aforementioned space $\Spin(\Pi)/\Spin(3)$ is a universal covering space of an appropriately chosen generalized flag variety. In general, we prove the following theorem:

\begin{thm*}
Let $\Pi$ be an irreducible Dynkin diagram such that the Bruhat decomposition of $G(\Pi)$ provides a CW decomposition (i.e., such that the conclusion of Proposition \ref{bruhat decomp is cw decomp} holds), let $I$ be the index set of the Dynkin diagram, let $J \subseteq I$, and let $P_J$ be a parabolic of type $J$. Then a presentation of $\pi_1(G/P_J)$ is given by 

\[ \left\langle x_i; \quad i \in I \mid x_ix_j^{\epsilon(i,j) } = x_jx_i,\quad x_k = 1; \quad  i,j \in I, k \in J \right\rangle.\]
In particular, this statement holds in the $2$-spherical and in the symmetrizable case.
\end{thm*}

We refer to \cite{Wig} for the analog result in the finite-dimensional situation.

In order to determine $\pi_1(K)$ in the general case, we compute subgroups of $\pi_1(K)$ corresponding to the index sets of connected components of $\Pia$ using the above theorem and covering maps of the type $K/K_J \to K/(K \cap T)K_J$ where $T$ is a maximal split torus of $G(\Pi)$ and $K_J$ is the subgroup of fixed points of a Levi factor of $P_J$ with both $T$ and $P_J$ invariant under the Cartan--Chevalley involution. We then show that $\pi_1(K)$ is a direct product of appropriately chosen such subgroups.

In a very similar way, the fundamental group of $\Spin(\Pi,\kappa)$ is determined, establishing the following theorem:

\begin{thm*}
Let $\Pi$ be an irreducible Dynkin diagram such that the Bruhat decomposition of $G(\Pi)$ provides a CW decomposition (i.e., such that the conclusion of Proposition 3.5 holds).
Let $n(g)$ be the number of connected components of $\Pia$ of colour $g$. Let $n(b,\kappa)$ be the number of connected components of $\Pia$ on which $\kappa$ takes the value 1 and which have colour $b$. Then 
\[\pi_1(\Spin(\Pi,\kappa)) \iso \ZZ^{n(g)} \times C_2^{n(b,\kappa)}.\] In particular, this statement holds in the $2$-spherical and in the symmetrizable case.
\end{thm*}

%
%It is not difficult to extend the methods of this article to the arbitrary two-spherical situation (see also Remark~\ref{next}). However, already in the spherical situation in the example $\Spin(C_n) = \mathrm{U}_1(\mathbb{C}) \times \mathrm{SU}_2(\mathbb{C})$ (see \cite[Remark~14.7]{GHKW}) one sees that $\pi_1(\Spin(C_n)) = \mathbb{Z}$. One concludes that for general diagrams $\Pi$ the fundamental group $\pi_1(\Spin(\Pi)/\Spin(3))$ is not necessarily trivial, which requires a more careful analysis than in the simply-laced case. In the general non-two spherical situation one additionally has to deal with the problem that the fundamental subgroups of rank two are not necessarily Lie groups.

\medskip
\textbf{Acknowledgements.} The research leading to this article has been partially funded by DFG via the project KO 4323/11. The authors thank Julius Gr\"uning and two anonymous referees for various helpful remarks on earlier versions of this article.

\section{Split-real Kac--Moody groups}
\label{Split-real Kac--Moody groups}
\label{Kac--Moody-basics}

%\begin{nota}
%In this article we denote isomorphism in the category of abstract groups by $\iso$.
%
%%For two topological spaces $U,V$, we will write $U \hom V$ to denote that $U$ is homeomorphic to $V$, while for two groups $G, H$, we will write $G \iso H$ to denote that $G$ is isomorphic to $H$.
%For $n \in \NN$, $\SO(n):= \SO(n,\RR)$ denotes the special orthogonal group over the real numbers.
%%For one- or two-element index sets $\{i,j\}$ we will use the convention $ij:= \{i,j\}$, e.g., $G_{ij} = G_{\{i,j\}} = G_{ji}$.
%
%%{\color{blue}Notation 2.2 finde ich komisch. Wie wird denn in der Literatur
%%zwischen diesen beiden Konzepten unterschieden? Tradition ist manchmal
%%praktisch, da Leute gleich wissen, wie sie etwas zu verstehen haben
%%(den Satz haettest Du jetzt vielleicht nicht erwartet von mir ...) --
%%ich schlage vor, dass wir versuchen, uns an die traditionelle Notation
%%zu halten.}\\
%%{\color{violet} Ich glaube nicht, dass es eine traditionelle Notation gibt (siehe dazu auch diesen Thread in einem Physikerforum: \\\url{https://www.physicsforums.com/threads/symbols-for-homeomorphic-isomorphic-homotopic.155073/}.\\
%%Meine Notation ist zwar unkonventionell und auch ziemlich krude, aber wenn man sie einmal gesehen hat, weiss man Bescheid und wird die beiden Symbole nicht verwechseln. Das kann dem Leser bei den Symbolen $\cong$ und $\simeq$ viel eher passieren, vor allem, wenn er selbst sie normalerweise anders verwendet. Aber du bist der Erfahrenere, also such du dir einfach eine Notation aus.}
%\end{nota}

In \cite[\S1.3]{Kac2}, Kac associates with every generalized Cartan matrix $\CA = (a_{ij})_{1\leq i,j \leq n} \in \ZZ^{n\times n}$ a quadruple $(\g_\CC(\CA), \h_\CC(\CA), \Psi, \check{\Psi})$ of a complex Lie algebra $\g_\CC(\CA)$, an abelian subalgebra $\h_\CC(\CA)$ and linearly independent finite subsets $\Psi = \{\alpha_1, \dots, \alpha_n\} \subseteq \h_\CC(\CA)^*$ and $\check{\Psi} = \{\check{\alpha}_1, \dots, \check{\alpha}_n\} \subseteq \h_\CC(\CA)$ called \emph{simple roots} and \emph{simple coroots}, respectively, such that $\alpha_j(\check{\alpha}_i) = a_{ij}$. Associated with such a quadruple is a Lie algebra generating set $\{e_1, \dots, e_n,f_1,\dots,f_n\} \cup \h_\CC(\CA)$. The complex Lie algebra $\g_\CC(\CA)$ is called the \emph{complex Kac--Moody algebra} associated with $\CA$, and $\h_\CC(\CA)$ its \emph{standard Cartan subalgebra}.

Since $a_{ij} \in \RR$, one can analogously define a quadruple $(\g_\RR(\CA), \h_\RR(\CA), \Psi, \check{\Psi})$ where $\g_\RR(\CA)$ is a real Lie algebra that embeds naturally into $\g_\CC(\CA)$ as the real form given by the involution induced by complex conjugation. One refers to $\g_\RR(\CA)$ as the \emph{split real Kac--Moody algebra} associated with $\RR$ and to $\h_\RR(\CA)$ as its \emph{standard split Cartan subalgebra}.

Let $Q  \subseteq \h_\RR(\CA)^*$ be the group generated by $\Psi$ and $Q_{\pm}$ the subsemigroups generated by $\pm \Psi$, respectively. For $k \in \{\CC, \RR\}$ and $\alpha \in \h_k(\CA)^*$ define the \emph{root space} 
\[\g_{\alpha}^k := \{X \in \g_k(\CA) \mid \forall H \in \h_k(\CA)^*: [H,X] = \alpha(H)X\}. \] The \emph{set $\Delta$ of $\h_k(\CA)$ roots in $\g_k(\CA)$} is defined as $\Delta:=\{\alpha \in Q \setminus \{0\} \mid \g_\alpha^k \neq \{0\}\}$. One has the \emph{root space decomposition} \[\g_k(\CA) = \h_k(\CA) \oplus \bigoplus_{\alpha \in \Delta} \g_\alpha^k.\]

The set $\Delta$ decomposes as a disjoint union into the subsets $\Delta_{\pm}:= \Delta \cap Q_{\pm}$ called {\em positive}, respectively {\em negative roots}. The restriction of the Lie bracket on $\g_\RR(\CA)$ to $$\mathfrak{u^\pm} := \bigoplus_{\alpha \in \Delta_\pm} \g_\alpha^k$$ turns $\mathfrak{u}^+$ and $\mathfrak{u}^-$ into Lie subalgebras of $\g_\RR(\CA)$. 

For $i = \{1,\dots, n\}$ define the \emph{fundamental root reflection} $\sigma_i \in \GL(\h_\RR(\CA)^*)$ by \[\sigma_i(\lambda):= \lambda -\lambda(\check{\alpha}_i)\alpha_i.\] Then the \emph{Weyl group} of $\g_\RR(\CA)$ is defined as $W:= \langle \sigma_1, \dots, \sigma_n\rangle \leq \GL(\h_\RR(\CA)^*)$ and forms a Coxeter system together with the set of fundamental root reflections. Finally, define the \emph{set of real roots} $\Phi:= W.\Psi \subseteq \Delta$ and $\Phi^{\pm} := \Delta_{\pm} \cap \Phi$, the {\em positive}, respectively {\em negative real roots}.

%\lila{\begin{de}
%\label{Kac--Moody group}
%In \cite[Theorem 1]{Ti}, Tits associates with every generalized Cartan matrix $\CA$ and every commutative ring $k$ a group $\calG_k(\CA)$. Let $\Pi$ be the Dynkin diagram of $\CA$. We set $G(\Pi) := [\calG_\RR(\CA),\calG_\RR(\CA)]$ and refer to this group as the \emph{algebraically simply-connected semisimple split real Kac--Moody group of type $\Pi$}.
%\end{de}}

The construction in \cite{Ti} of $\calG_\CA(\RR)$ (see Definition \ref{Kac--Moody group}) provides a representation of $\calG_\CA(\RR)$ on $\g_\RR(\CA)$ by Lie algebra automorphisms, which is denoted by \[ \Ad: \calG_\CA(\RR) \to \Aut(\g_\RR(\CA)),  \] and referred to as the \emph{adjoint representation} of $\calG_\CA(\RR)$. Since the subgroup $\Ad(G(\Pi))$ of $G(\Pi)$ under this representation preserves the commutator subalgebra $\g_\RR'(\CA)$, one obtains an adjoint representation \[ \Ad: G(\Pi) \to \Aut(\g_\RR'(\CA))  \] for $G(\Pi)$. The kernels of the adjoint representations of $\calG_\CA(\RR)$ and $G(\Pi)$ are given by the respective centres.

An element $X \in \g_\RR(\CA)$ is \emph{$\ad$-locally-finite} if for every element $Y \in \g_\RR(\CA)$ there exists an $\ad(X)$-invariant finite-dimensional subspace $W$ with $Y \in W$. As pointed out in \cite[p. 64]{Mar}, this implies that $\left.\ad(X)\right|_W$ is a (finite) matrix in some basis of $W$, so the exponential $\exp(\ad(X))$ can be defined in the ususal way.
By \cite[(KMG5), p.~545]{Ti} and the uniqueness properties of $G_\CA(\RR)$ established in \cite[Theorem 1]{Ti}, $\exp(\ad(X)) \in \Ad(G_\CA(\RR))$. Let $F_{\g_\RR(\CA)}$ and $F_{\g_\RR'(\CA)}$ be the subsets of $\ad$-locally-finite elements of the respective algebras. 
The maps $\exp: F_{\g_\RR(\CA)} \to \Ad(\calG_\CA(\RR))$ and $\exp: F_{\g_\RR'(\CA)} \to \Ad(G(\Pi))$ given by $X\mapsto \exp(\ad(X))$ can be lifted to exponential functions $\exp: F_{\g_\RR(\CA)} \to \calG_\CA(\RR)$ and $\exp: F_{\g_\RR'(\CA)} \to G(\Pi)$.

For $X \in \h_\RR(\CA)\subseteq F_{\g_\RR(\CA)}$, one has 
\begin{eqnarray}
 \mathrm{ad}(X)(e_i) =  [X,e_i] = \alpha_i(X) e_i,& \quad\quad &   \mathrm{Ad}(\exp(X))(e_i) = e^{\alpha_i(X)}\cdot e_i,\label{AdRootSpaces} \\
 \mathrm{ad}(X)(f_i) =  [X,f_i] = -\alpha_i(X) f_i,& \quad\quad &   \mathrm{Ad}(\exp(X))(e_i) = e^{-\alpha_i(X)}\cdot f_i, 
\end{eqnarray}
cf.\ \cite[Section~6.1.6]{Kum}, \cite[(KMG5), p.~545]{Ti}.
%\rot{$e_i$ und $f_i$ muessen vorher noch eingefuehrt werden!}

The same constructions apply also to $\CC$ instead of $\RR$. Since $\h_\CC(\CA) \subseteq F_{\g_\CC(\CA)}$, one can define $T_\CC:= \exp(\h_\CC(\CA)
)$. Note that  $\exp(\h_\RR(\CA)) =: A_\RR \subsetneq T_{\RR} := T_\CC \cap \calG_\CA(\RR)$. There is a unique Lie group topology on $T_\RR$ in which $T_\RR \iso (\RR^\times)^n$ and $A_\RR = T_\RR^\circ \iso (\RR_{>0})^n$. The centre of $\calG_\CA(\RR)$ is contained in $T_\RR$.  

The intersection $T:= G(\Pi) \cap T_\CC$ is called the \emph{standard split maximal torus} of $G(\Pi)$; again, $A_\RR\cap T$ is of finite index in $T$ and $T$ contains the centre of $G(\Pi)$.

The Lie algebra $\mathfrak g_\RR({\mathbf A})$ admits a unique involution $\theta$ which maps $e_j$ to $f_j$ for all $j=1, \dots, r$ and acts as $-1$ on $\mathfrak h_\RR({\mathbf A})$. There exists a unique involutive automorphism $\theta: G_\CA(\RR) \to G_\CA(\RR)$ such that $\theta(\exp(X)) = \exp(\theta(X))$ for all $X \in F_{\mathfrak g_\RR({\mathbf A})}$, and this involutive automorphism is called the \emph{Cartan--Chevalley involution} of $G_\CA(\RR)$. We denote by $K_\CA(\RR) := G_\CA(\RR)^\theta \subset G_\CA(\RR)$ the fixed point subgroup of this involution and define $K(\Pi) := K_\CA(\RR) \cap G(\Pi)$.

Let $\alpha \in \Phi$ be a real root. Then $\g_\alpha^\RR$ is one-dimensional and consists of ad-locally-finite elements. One can therefore define the \emph{root group} $U_\alpha:= \exp(\g_\alpha^\RR) \subseteq \calG_\CA(\RR)$. Each root group $U_\alpha$ carries a unique Lie group topology such that $U_\alpha \iso \RR$ as topological groups. Root groups corresponding to positive real roots are called {\em positive root groups}, root groups corresponding to negative real roots are called {\em negative root groups}.

Define the \emph{positive}, respectively \emph{negative maximal unipotent subgroup} $U^\pm$ of $\calG_\CA(\RR)$ as the group generated respectively by the positive and negative root groups. One has $U^\pm \subseteq G(\Pi)$. The groups $U^\pm$ are normalized by $T_\RR$ and intersect $T_\RR$ trivially. In particular, they intersect the centres of $\calG_\CA(\RR)$ and $G(\Pi)$ trivially and hence embed into both $\Ad(\calG_\CA(\RR)$ and $\Ad(G(\Pi))$.

If $\alpha \in \Phi^+$, then $-\alpha \in \Phi^-$ and the group $G_\alpha:= \langle U_\alpha, U_{-\alpha} \rangle \leq G(\Pi)$ is isomorphic to $\SL_2(\RR)$.  The groups $G_\alpha$ with $\alpha \in \Phi^+$ are called the \emph{rank-$1$ subgroups} and the groups $G_1:=G_{\alpha_1}, \dots, G_n:=G_{\alpha_n}$ are called the \emph{fundamental rank-$1$ subgroups} of $G(\Pi)$.

One can show that the pair $((U_\alpha)_{\alpha \in \Phi}, T)$ defines an RGD system for $G(\Pi)$. For details concerning RGD systems, we refer the reader to \cite[Chapter~8]{AB}.

Recall that the generalized Cartan matrix $\CA$ is called {\em $2$-spherical}, if $a_{ij}a_{ji} \leq 3$ for all $i \neq j \in I$; in other words, if the orders of the products $\sigma_i \sigma_j$ are always finite. The generalized Cartan matrix $\CA$ is {\em symmetrizable}, if it is the product of a symmetric and a diagonal matrix. These notions are also applied to any and all objects that are derived from $\CA$ such as the (extended) Weyl group, the Kac--Moody group, their buildings, etc.

\begin{de+rem}
\label{definition kp}
The \emph{Kac--Peterson topology} on $\calG_\CA(\RR)$ equals the finest group topology on $\calG_\CA(\RR)$ such that the natural embeddings $(U_\alpha \hookrightarrow \calG_\CA(\RR))_{\alpha \in \Phi}$ and $T_\RR \hookrightarrow \calG_\CA(\RR)$ are continuous when $T_\RR$ and the root groups $U_\alpha$ are endowed with their Lie group topologies.

The Kac--Peterson topology is $k_\omega$ by \cite[Proposition 7.10]{HKM} and, in particular, Hausdorff. Moreover, for every $\alpha \in \Phi^+$, it induces the unique connected Lie group topology on $G_\alpha$ and on $T_\RR$ by \cite[Corollary 7.16]{HKM}
\end{de+rem}
For more details on the Kac--Peterson topology, see \cite[Chapter~7]{HKM}.

%From now on, we will always consider $G(\Pi)$ as a topological group with respect to the Kac--Peterson subspace topology.

\begin{nota}
Throughout this paper, let $G:= G(\Pi):=[\calG_\CA(\RR),\calG_\CA(\RR)]$ be the algebraically simply connected semisimple split real Kac--Moody group associated to an irreducible generalized Dynkin diagram $\Pi = (V,E)$ with (bijective) labelling $\lambda:\{1, \dots, n\} =: I \to  V $. Let $K := K(\Pi)$ be the maximal compact subgroup of $G$, i.e., the subgroup fixed by the Cartan--Chevalley involution $\theta$. 

Denote by $B:= B_+$ the positive Borel subgroup of the twin $BN$-pair of $G$, by $T$ the standard split maximal torus and by $W$ the Weyl group of $G$ with generating set $S = \{\sigma_i\}_{i \in I}$. For each $\sigma_i \in S$, take $s_i \in G$ to be a fixed representative of order $4$ for $\sigma_i$. The group $\tilde{W} := \langle s_i \mid i \in I \rangle \leq G$ is called the {\em extended Weyl group}. By \cite[Corollary 1.7]{DMGH}, one has an Iwasawa decomposition $G = KB$.  
%Hier kann stattdessen der Anhang zitiert werden, allerdings mit Erklaerung, dass $KAU = KB$ ist.

The groups $G$ and $K$ are always endowed with the subspace topologies induced by the Kac--Peterson topology on $G_\CA(\RR)$ and $G/B$ with the quotient topology.

Unless specified more explicitly, the symbol $J$ will always denote an arbitrary subset of the index set $I$, the symbol $\Pi_J$ the subdiagram of $\Pi$ corresponding to $J$, the symbol $G_J$ the subgroup $G(\Pi_J)$ of $G$, and the symbols $K_J$ and $B_J$ the intersections $G_J \cap K$ and $G_J \cap B$, respectively. This is consistent with the notation for the fundamental rank one subgroups: One has $G(\Pi_i) = G_i = G_{\alpha_i}$.

%In particular, for $i \neq j \in I$, the groups $G_i$ and $G_{ij}$ are, respectively, fundamental rank one and rank two subgroups of $G$. One has $G_i \iso \SL(2,\RR)$.
\end{nota}

\begin{rem}
\label{homeos alpha beta}
Due to the structure theory of RGD systems (cf.\ \cite[Chapter~8]{AB}, most notably the fact that restricting an RGD system to a subdiagram again yields an RGD system), for each fundamental rank one subgroup $G_{i}$  there exists an (abstract) isomorphism $\gamma_{i}: \SL(2,\RR) \to G_{i}$ with the following properties: Let $B_{\SL(2,\RR)}$ be the group of upper triangular matrices in $\SL(2,\mathbb{R})$ and let $U_{\pm \beta}$ denote the canonical root subgroups of $\SL(2,\RR)$. Then
\begin{itemize}
%\item $\gamma_{ij}(G_i) =  \epsilon_{(i,j)}(\SL(2,\RR))$, $\beta_{kl}(G_k) =  \iota_{(k,l)}(\SL(2,\RR))$.
%\item $\gamma_{ij}(s_i) =  \epsilon_{(i,j)} ({s})$, $\beta_{kl}(s_k) =  \iota_{(k,l)} ({s})$ 
\item $\gamma_{i}(U_{\pm \beta}) = U_{\pm \alpha_{i}}$.                    
\item $\gamma_{i}(B_{\SL(2,\RR)}) = B_{i}$.
\item For each $x \in \SL(2,\RR)$, $\gamma_{i}((x^t)^{-1}) = (\gamma_{i}(x))^{\theta}$,  and hence
\item $\gamma_{i}(\SO(2,\RR)) =  K_i$.
\end{itemize}

By \cite[Corollary~7.16]{HKM}, the restriction of the Kac--Peterson topology to any spherical subgroup $H$ of $G$ coincides with its Lie topology. That is, the groups $G_{i}$ inherit their Lie group topology from the topological Kac--Moody group $G$.
%Since the groups $G_{i}$ inherit their Lie group topology from $G_\RR(\CA)$, 
By the classical theory of Lie groups this yields the existence of a diffeomorphism $\gamma_i$ with the desired properties; in particular, $\gamma_{i}$ is an open map.
%By the open mapping theorem for locally compact groups (e.g.\ \cite[Theorem~6.19]{Str}) these continuous isomorphisms are homeomorphisms.
\end{rem}

\begin{de}
\label{basics}
%Let $(G, (U_\alpha)_{\alpha \in \Phi}, T)$ be an RGD system with Tits system $(G,B_+,N,S)$ and Weyl group $W$, let $B:= B_+$, 
Using the Bruhat decomposition $G = \bigsqcup_{w \in W} BwB$ (\cite[Theorem~6.56, Remark (1)]{AB}), let
\begin{eqnarray*}
\delta: G/B \times G/B & \to & W \\ \delta(gB,hB)  =  w & \iff & g^{-1}h \in BwB
\end{eqnarray*}
be the Weyl distance function on $G/B$, and let $l_S$ be the length function that associates to each element the (unique) length of a corresponding reduced expression in $S$. Let $\leq$ be the strong Bruhat order on $W$: Recall that for $w_1, w_2 \in W$ one has $w_1 \leq w_2$ if there exist reduced expressions $s_{i_1}\dots s_{i_{l_S(w_1)}}$ of $w_1$ and $s_{j_1}\dots s_{j_{l_S(w_2)}}$ such that the former is a (not necessarily consecutive) substring of the latter.

 For $w \in W$ and a chamber $gB \in G/B$ define $$C_w(gB)  := \{hB \in G/B\mid \delta(gB, hB) = w\},$$  $$C_{\leq w}(gB) := \bigcup_{v \leq w}C_v(gB)$$ and $$C_{<w}(gB) :=C_{\leq w}(gB) \setminus C_w(gB).$$  In particular, one has $C_w(B)= BwB/B$ and $C_{\leq \sigma}(B) = B \langle s \rangle B/B$ for $\sigma \in S$ with representative $s \in \tilde{W}\subseteq G$ in the extended Weyl group $\tilde{W}$. A set $C_{\leq \sigma}(gB)$ is called a \emph{$\sigma$-panel}.

 Moreover, for a subset $\{\sigma_i\}_{i \in J} \subseteq S$ with representatives $\{s_i\}_{i \in J} \subseteq \tilde{W}$ define $P_J$ to be the standard parabolic subgroup corresponding to the index set $J$, that is, $P_J := B\langle \{s_i\}_{i \in J} \rangle B$.   
\end{de}

Throughout this paper, $C_w(gB)$  and $C_{\leq w}(gB)$ will always be endowed with the subspace topologies induced by $G/B$.

\begin{lem}
\label{P_i = G_iB}
Let $\sigma_i \neq \sigma_j \in S$.
% and let $G_i := \langle U_{\alpha_{s_i}}, U_{-\alpha_{s_i}}\rangle$ be the rank one subgroup corresponding to $s_i$. 
Then the following hold:
\begin{enumerate}
%\item $Bs_iB = U_{\alpha_{s_i}}s_iB$ \rot{wirklich benoetigt?}
\item $P_i = G_i B$ $= K_i B$. In particular, $C_{\leq \sigma_i}(B) = K_iB/B$.
\item $Bs_i s_jB = Bs_iB Bs_jB$. In particular, $C_{\leq \sigma_i \sigma_j}(B) = K_iBK_jB/B$ %= U_{\alpha_{s_i}} U_{\sigma_i(\alpha_{s_j})} s_i s_j B$. \rot{(ist $\sigma_i(\alpha_{s_j})$ der korrekte Ausdruck?}
\end{enumerate}
\begin{proof} Assertions (a) and (b) follow from \cite[Remark~8.51]{AB} and \cite[Remark (2) after Theorem~6.56]{AB}, respectively, and the Iwasawa decomposition $G_i = K_iB_i$.    
\end{proof}
\end{lem}

%\section{Conventions, Basics, and Notation}
%%\section{Basics and Notation}
%\input{conventions_basics_notation.tex}

%\section{Four homeomorphisms and a covering map}
%\input{Gij_Bij.tex}

\section{The fundamental group of the generalized flag variety \texorpdfstring{$G/P_J$}{G/PJ}}

For a moment, let $\Pi$ be an irreducible simply-laced diagram distinct from $A_1$, and let $G = G(\Pi)$ and $K = K(\Pi)$ be as in the preceding section. Moreover, let $\mathrm{Spin}(\Pi)$ be the double cover of $K(\Pi)$ constructed in \cite[Lemma 16.18]{GHKW} (see Definition~\ref{Spin remark} below). By construction, any $A_2$-subdiagram of $\Pi$ yields an embedding $\mathrm{Spin}(3) \hookrightarrow \mathrm{Spin}(\Pi)$ and, since $\mathrm{Spin}(3)$ inherits the Lie topology from the Kac--Peterson topology on $\mathrm{Spin}(\Pi)$ by \cite[Corollary~7.16]{HKM}, one obtains a locally trivial fibre bundle $$\Spin(3) \to \Spin(\Pi) \to \Spin(\Pi)/\Spin(3)$$ by \cite{Pal} (see Proposition~\ref{Palais}). It will turn out in Section~\ref{4} below that $\Spin(\Pi)/\Spin(3)$ is a universal covering space of the generalized flag variety $G/P_J$ where $J \subset I$ equals the set consisting of the two types involved in the chosen $A_2$-subdiagram. The fundamental group of $\Spin(\Pi)$ then follows from the homotopy exact sequence \[ \{1\} = \pi_1(\Spin(3)) \rightarrow \pi_1(\Spin(\Pi)) \rightarrow \pi_1(\Spin(\Pi)/\Spin(3)) = 1.\] This motivates our interest in the fundamental group and covering theory of {\em generalized flag varieties} $G/P_J$.

\medskip \noindent
Throughout this section, let $J \subseteq I$, let $W_J$ be the subgroup of $W$ generated by $\{\sigma_i\}_{i \in J}$, and let $W^J \subseteq W$ be a set of representatives of the cosets in $W/W_J$ that have minimal length in the coset they define.

\begin{lem}[(Bruhat decomposition)]
\label{bruhat}
One has $G/P_J = \bigsqcup_{w\in W^J}BwP_J/P_J$.
\end{lem}

\begin{proof}
This follows immediately from \cite[Theorem~6.56, Remark (1)]{AB}.
%By \cite[Theorem~6.56, Remark (1)]{AB}, the group $G$ has a Bruhat decomposition $$G = \bigsqcup_{w \in W} BwB = \bigcup_{w \in W} BwBW_JB =  \bigcup_{w \in W} BwP_J.$$
%Since double cosets partition $G$, one has $w_1 W_J = w_2 W_J$ if and only if $Bw_1 P_J = B w_2 P_J$ for $w_1, w_2 \in W$. This yields the desired disjoint decomposition of $G/P_J$.   
\end{proof}

\begin{lem}
\label{quotient basics}
Let $G$ be a topological group and $H_1 \leq H_2$ subgroups of $G$ and endow $G/H_i$ with the quotient topology. Then the following hold:
\begin{enumerate}
\item The projection map $\pi: G \to G/H_1$ is continuous and open.
\item The canonical map $\psi: G/H_1 \to G/H_2$ is continuous and open.
\end{enumerate}
\end{lem}

\begin{proof}
This is a standard exercise for topological groups.
%(a): Let $U \subseteq G$ open. Since $\pi$ is a quotient map, it suffices to show that $\pi^{-1}(\pi(U))$ is open. But this is true since $\pi^{-1}(\pi(U)) = UH_1 = \bigcup_{h \in H_1}Uh$ is a union of translates of open sets.
%
%
%(b): This follows directly from (a) and the commutative diagram
%(b): Consider the commutative diagram
%\[\begin{tikzcd}
%G \arrow[rd, "\phi"] \arrow[d, "\pi"]  \\
%G/H_1 \arrow[r, "\psi"] & G/H_2 
%\end{tikzcd}.\qedhere\] 
%Let $UH_2 \subseteq G/H_2$ be open. Then $\pi^{-1}(\psi^{-1}(UH_2)) = \phi^{-1}(UH_2) $ is open in G, hence by (a),   $\psi^{-1}(UH_2) = \pi(\phi^{-1}(UH_2))$ is open in $G/H_1$. 
%
%Let $VH_1 \subseteq G/H_1$ be open. Then $\pi^{-1}(VH_1) = VH_1$ is open in $G$. Since $$\phi^{-1}(\psi(VH_1)) = \phi^{-1} (VH_2) = VH_2= VH_1 H_2 = \bigcup_{h \in H_2} VH_1 h,$$ the preimage  $\phi^{-1}(\psi(VH_1)) $ is open and therefore $\psi(VH_1)$ is open in $G/H_2$. 
\end{proof}

\begin{de+rem}
\label{definition psi_w}
For $w \in W$, define the following restrictions of the canonical map $\psi: G/B \to G/P_J$:\begin{itemize}
	\item $\psi_w: BwB/B \to BwP_J/P_J$,
	\item $\psi_{\bar{w}}: \bigcup_{x \leq w} BxB/B \to \bigcup_{x \leq w} BxP_J/P_J$.
\end{itemize}
Since $\psi$ is continuous, the same holds for the two restrictions. The space $\bigcup_{x \leq w} BxB/B$ is compact by \cite[Corollary~3.10]{HKM} and so $\psi_{\bar{w}}$ is a quotient map.
\end{de+rem}

\begin{lem}
\label{hom between BwB/B and BwP/P}
Let $G$ be $2$-spherical or symmetrizable and let $w \in W^J$. Then the canonical map $\psi_w$ is a homeomorphism.
\end{lem}
%\rot{Renamed $\psi$ to $\psi_w$, and the canonical map $G/B \to G/P_J$ to $\psi$, needs to be changed throughout the document}.
\begin{proof}

By Remark \ref{definition psi_w}, $\psi_{\bar{w}}$ is a quotient map. 
One has $\psi_{\bar{w}}^{-1}(BwP_J/P_J) = BwB/B$: Let $x \leq w$ such that $BxP_J/P_J = BwP_J/P_J$. Then $x\in BwP_J = BwW_JB$ where the equality holds since by definition of $W^J$ one has $l(ww')= l(w)+l(w')$ for all $w' \in W_J$ which implies $Bww'B = BwBw'B$. The Bruhat decomposition of $G$ yields $x \in wW_J$ and hence, $l(x) \geq l(w)$. This implies $x = w$.

Now, since $BwB/B$ is open in its closure $\bigcup_{x \leq w} BxB/B$ in $G/B$ (see \cite[Proposition~5.9]{HKM} plus Corollary~\ref{topstrong}), the preceding observations yield that $\psi_w$ is an injective quotient map and therefore a homeomorphism.
\end{proof}

\begin{lem}
\label{panels are spheres}
Let $\sigma_i \in S$. Then each panel $C_{\leq \sigma_i}(B)$ is homeomorphic to the $1$-sphere $\mathbb{S}^1$.
\end{lem}

\begin{proof}
The panel $C_{\leq \sigma_i}(B)$ is a subbuilding of $G/B$ corresponding to the RGD system $\{G_i, U_{\alpha_{i}}, U_{-\alpha_{i}}, T \cap G_i \}$. By Remark~\ref{homeos alpha beta} one has $G_i \iso \SL(2,\RR)$, $T \cap G_i \iso T_{\SL(2,\RR)}$ and $U_{\pm\alpha_{i}} \iso U_{\pm\alpha}$ where $T_{\SL(2,\RR)}$ denotes the subgroup of diagonal matrices and $U_{\pm\alpha}$ denote the canonical root subgroups of $\SL(2,\RR)$. This implies that $C_{\leq \sigma_i}(B)$ is homeomorphic to the building $\SL(2,\RR)/B_{\SL(2,\RR)} \hom \PP_1(\RR) \hom \mathbb{S}^1$.
\end{proof}

\begin{de}
Following \cite[Chapter 8]{Rot}, a \emph{CW complex} is an ordered triple $(X, E, \chi)$, where $X$ is a Hausdorff
space, $E$ is a family of cells in $X$, and $\chi = \{\chi_e\mid e\in E\}$ is a family of maps, such
that
\begin{enumerate}
	\item $X = \bigsqcup_{e\in E}{E}$.
\item For $k \in \NN$, let $X^{(k)}\subseteq X$ be the union of all cells of dimension $\leq k$. Then for each $(k+1)$-cell $e \in E$, the map $\chi_e:(D^{k+1}, S^{k}) \to (e\cup X^{(k)}, X^{(k)})$, is a \emph{relative
homeomorphism}, i.e., it is a continuous map and its restriction $D^{k+1} \setminus S^{k} \to e$ is a homeomorphism.
\item If $e \in E$, then its closure $\cl{e}$ is contained in a finite union of cells in $E$.
\item $X$ has the weak topology determined by $\{\cl{e} \mid e \in E\}$, i.e., a subset $A$ of $X$ is closed if and only if $A \cap \cl{e}$ is closed in $\cl{e}$ for each $e \in E$.
\end{enumerate}

For $k \in \NN$, let $\Lambda_k$ be an index set for the $k$-dimensional cells, so that $X^{(k)} \setminus X^{(k-1)} = \bigsqcup_{\lambda \in \Lambda_k} e_\lambda$ and set $\chi_\lambda:= \chi_{e_\lambda}$. This map is called the \emph{characteristic map} of $e_\lambda$.
\end{de}

\begin{prop}
\label{bruhat decomp is cw decomp}
Let $G$ be $2$-spherical or symmetrizable. Then for each $w \in W$, the set $C_w(B) = BwB/B$ is a cell of dimension $l(w)$ that is open in its compact closure $C_{\leq w}(B)$ in $G/B$. For each subset $J \subseteq I$, the Bruhat decomposition $G/P_J = \bigsqcup_{w\in W^J}BwP_J/P_J$ is a CW decomposition. 
\end{prop}

\begin{proof}
The first statement is immediate by \cite[Corollary~3.10 and Proposition~5.9]{HKM} plus Corollary~\ref{topstrong}, see also \cite[p. 170, 171]{Kra}. Furthermore, \cite[Proposition~5.9]{HKM} combined with Corollary~\ref{topstrong} states that the Bruhat decomposition of $G/B$ is a CW decomposition. By Lemma \ref{hom between BwB/B and BwP/P}, $G/P_J$ is composed of cells that are homeomorphic to cells in $G/B$, so composing the characteristic maps of the latter cells with the canonical map $\psi: G/B \to G/P_J$ yields characteristic maps for the cells in $G/P_J$. 

For the closure-finiteness, let $BwP_J/P_J$ be a cell in $G/P_J$. Since $\psi$ is continuous and restricts to a homeomorphism $BwB/B \to BwP_J/P_J$, it maps  
$\cl{BwB/B}$ surjectively onto $\cl{BwP_J/P_J}$. Now, $\cl BwB/B = \bigcup_{x \leq w} BxB/B$, which implies that 
$$\cl{BwP_J/P_J} = \bigcup_{x \leq w} BxP_J/P_J = \bigcup_{\substack{x\leq w \\ x \in W^J}} BxP_J/P_J,$$ where the last equality holds since $W_J \subseteq P_J$. This proves that $\cl{BwP_J/P_J}$ is contained in a finite union of cells.

It remains to show that $G/P_J$ has the weak topology determined by the cell closures.

For $w\in W$ and a minimal-length representative  $\tilde{w}\in W^J$ of $wW_J$, one has $BwP_J/P_J = B\tilde{w}P_J/P_J$. Let $e_w:= BwP_J/P_J = B\tilde{w}P_J/P_J$ and $e_w':= BwB/B$. Let $\bar{e}_w = \cl{e_w} = \bigcup_{x \leq \tilde{w}} BxP_J/P_J$ and $\bar{e}_w':= \cl{e'_w} =  \bigcup_{x \leq w} BxB/B$.

%and let $\bar{e}':= \cl{e'} =  \bigcup_{x \leq w} BxB/B$ and $\bar{e}:= \cl{e} =  \bigcup_{x \leq w} B x P_J/P_J$. 

 %For $e = BwP_J/P_J \in E$, let $e' := BwB/B$ and let $\bar{e}':= \cl{e'} =  \bigcup_{\tilde{w} \leq w} B\tilde{w}B/B$ and $\bar{e}:= \cl{e} =  \bigcup_{\tilde{w} \leq w} B\tilde{w}P_J/P_J$. Define the restriction  
%$\psi_{\bar{e}}: \bar{e}'\to \bar{e}$ of $\psi$. Then 
%%$\psi_e$ is a homeomorphism by Lemma \ref{hom between BwB/B and BwP/P} and 
%$\psi_{\bar{e}}$ is a quotient map since $e'$ is compact by \cite[Corollary~3.10]{HKM}.

Let $A$ be a closed subset of $G/P_J$ and let $e_w$, $w \in W^J$, be an arbitrary cell. Then $\psi^{-1}(A)$ is closed in $G/B$ since $\psi$ is continuous, so $\psi^{-1}(A) \cap \bar{e}_w'$ is closed in $\bar{e}_w'$ since $G/B$ is a CW complex. Now, \[ \psi^{-1}(A) \cap \bar{e}_w' = \psi^{-1}(A) \cap \psi^{-1}(\bar{e}_w) = \psi^{-1}(A \cap \bar{e}_w) = \psi_{\bar{w}}^{-1}(A \cap \bar{e}_w).\] Since $\psi_{\bar{w}}$ is a quotient map by Remark \ref{definition psi_w}, this implies that $A \cap \bar{e}_w$ is closed in $\bar{e}_w$.

Now, let $A$ be a subset of $G/P_J$ such that $A \cap \bar{e}_w$ is closed in $\bar{e}_w$ for all $w \in W^J$. Since for each $w \in W$ one has $e_w = e_{\tilde{w}}$ for any minimal-length representative $\tilde{w} \in W^J$ of $wW_J$, in fact $A \cap \bar{e}_w$ is closed in $\bar{e}_w$ for all $w \in W$. Therefore $\psi_{\bar{w}}^{-1}(A  \cap \bar{e}_w)$ is closed in $\bar{e}_w'$ for all $w \in W$. Since $\psi_{\bar{w}}^{-1}(A  \cap \bar{e}_w) = \psi^{-1}(A) \cap  \bar{e}_w'$, the fact that $G/B$ is a CW complex implies that $\psi^{-1}(A)$ is closed in $G/B$. Since $\psi$ is open by Lemma \ref{quotient basics}, it follows that $A$ is closed in $G/P_J$. This proves that $G/P_J$ is a CW complex.
 %\rot{To show that the Bruhat decomposition is a CW decomposition, one has to verify that} 
%By Lemma \ref{hom between BwB/B and BwP/P}, it suffices to prove that for $w \in W^J$ the relative boundary of the closure $\cl BwP_J/P_J$ of a cell $BwP_J/P_J$ is contained in a finite union of cells of smaller dimension. The canonical map  $G/B \to G/P_J$ maps  
%the relative boundary $\cl BwB/B$ surjectively onto $\partial (\cl BwP_J/P_J)$. Now, $\partial (\cl BwB/B) = \bigcup_{\tilde{w} < w} B\tilde{w}B/B$, which is mapped to $$\bigcup_{\tilde{w} < w} B\tilde{w}P_J/P_J = \bigcup_{\substack{\tilde{w} < w \\ \tilde{w} \in W^J}} B\tilde{w}P_J/P_J,$$ since $W_J \subseteq P_J$. This proves the claim.
\end{proof}

The preceding result combined with the following lemma (which is a consequence of \cite[Ch. 7, Thm 2.1]{Mas}) will allow us to efficiently compute the fundamental group of a generalized flag variety in Theorem~\ref{fundamental group} below.

\begin{lem}
\label{presentation cw}
Let $X$ be a $CW$ complex with only one $0$-cell $x_0$. For each $\lambda \in \Lambda_2$, let $f_\lambda:[0,1] \to \mathbb{S}^1$ be a loop whose homotopy class generates $\pi_1(\mathbb{S}^1)$ and whose image $\gamma_\lambda:= \chi_\lambda \circ f_\lambda$ under $\chi_\lambda$ is a loop in $X^{(1)}$ starting at $x_0$. Then \[ \left \langle [\chi_\mu], \quad \mu \in \Lambda_1 \mid [\gamma_\lambda], \quad \lambda \in \Lambda_2 \right \rangle\] is a presentation of $\pi_1(X,x_0)$, where the brackets denote the respective homotopy classes in $X^{(1)}$.
\end{lem}

%\begin{rem}
%The results from \cite{HKM} cited in the above proof and in the proof of Lemma \ref{hom between BwB/B and BwP/P} rely on the fact that the twin building associated to a two-spherical simply connected split real Kac--Moody group is a strong topological twin building. This, in turn, is a consequence of the openness of the map $m: U_+ \times T \times U_{-} \to B_+B_{-}:(u_+,t,u_{-}) \mapsto u_+tu_{-}$, see \cite[Proposition~7.31]{HKM}. This openness result has been announced without proof in \cite[Theorem~4]{Kac1983} without restriction to the two-spherical case. As \cite{HKM} point out in a remark to Proposition~7.31, this more general statement would allow one to remove the requirement that $G$ be two-spherical from Proposition~7.31. Consequently, two-sphericity could then be removed from Proposition \ref{bruhat decomp is cw decomp} and all further results here; in particular Theorem \ref{fundamental group of K} and Theorem \ref{fundamental group of Spin}
%\end{rem}

Next, we study the characteristic maps of the CW decomposition of a generalized flag variety explicitly.

\begin{nota}
Define $R: [0,1] \to \SO(2,\RR), s \mapsto \begin{pmatrix} \cos(s\pi) & -\sin(s\pi) \\ \sin(s\pi) & \cos(s\pi) \end{pmatrix}.$	
%Let $D:= R([0,1]) \subseteq \SO(2)$.
\end{nota}

\begin{lem}
\label{the map R}
$R$ induces a continuous, surjective map $\tilde{R}: [0,1] \to \SL(2,\RR)/B_{\SL(2,\RR)}$ which maps the interior $(0, 1)$ homeomorphically onto its image and maps the boundary $\{0, 1\}$ surjectively onto its image.
\end{lem}

\begin{proof}
Let $\{x_0\}:= \left\langle \begin{pmatrix} 1 & 0 \end{pmatrix}^\intercal \right\rangle \in \PP^1$ where $\PP^1$ denotes the real projective line, modelled as the subset of one-dimensional subspaces of $\RR^2$. Since each one-dimensional subspace in $\PP^1 \setminus \{x_0\}$ contains exactly one element in the upper half circle $R([0,1]) \cdot \begin{pmatrix} 1 & 0 \end{pmatrix}^\intercal$  while $x_0$ contains the two boundary points corresponding to $R(0)$ and $R(1)$, one has a surjection from $[0,1]$ onto $\PP^1$ given by $t \mapsto \left\langle R(t) \cdot  \begin{pmatrix} 1 & 0 \end{pmatrix}^\intercal \right\rangle$ which maps $(0,1)$ bijectively onto $\PP^1 \setminus \{x_0\}$. Since $\SL(2,\RR)$ acts transitively on the real projective line $\PP^1$ with $B_{\SL(2,\RR)}$ being the stabilizer of $x_0:= \left\langle \begin{pmatrix} 1 & 0 \end{pmatrix}^\intercal \right\rangle$, one has a bijective correspondence $gB \mapsto g x_0$ between $\SL(2,\RR)/B_{\SL(2,\RR)}$ and $\PP^1$. This yields the desired surjectivity and bijectivity properties of $\tilde{R}$. Continuity is clear, as well as the fact that the restriction to the interior is a homeomorphism.
\end{proof}

\begin{de}
Let $D^1 = [0,1]$ be the $1$-dimensional unit disc and note that $D^2 \hom D^1 \times D^1$. For $i,j \in I$ let $\gamma_i, \gamma_j$ be as in Remark \ref{homeos alpha beta}.  Let $p:G \to G/B$ be the canonical projection. Define $\chi_i: D^1 \to G/B$ and $\chi_{(i,j)}: D^1 \times D^1 \to G/B$ by \begin{itemize}
	\item $\chi_i(s) := p(\gamma_i(R(s))) = \gamma_i(R(s))\cdot B$,
	\item $\chi_{(i,j)}(s,t):= p(\gamma_i(R(s))\gamma_j(R(t))) = \gamma_i(R(s))\gamma_j(R(t))\cdot B$.
\end{itemize}

\end{de}
The following lemma was inspired by \cite[Ch. 10, second Proposition of 6.8]{Pro}, see also \cite[\S2.6, p.~198]{Kac}.

\begin{lem}
\label{char maps}
Let $G$ be $2$-spherical or symmetrizable. Then the maps defined above are characteristic maps for the following cells:
\begin{enumerate}
\item $\chi_i$ for $C_{{\sigma}_i}(B) = B {s}_i B /B $,
%\item $\zeta_i$ for $C_{{\sigma}_i}(B_{\SL(2,\RR)\times \SL(2,\RR)}) = B_{\SL(2,\RR)\times \SL(2,\RR)} {s}_i B_{\SL(2,\RR)\times \SL(2,\RR)} /B_{\SL(2,\RR)\times \SL(2,\RR)}$,
\item $\chi_{(i,j)}$ for  $C_{{\sigma}_i {\sigma}_j}(B) = B {s}_i {s}_j B /B$.
%\item $\zeta_{(i,j)}$ for $C_{{\sigma}_i {\sigma}_j}(B_{\SL(2,\RR)\times \SL(2,\RR)}) = B_{\SL(4)} {s}_i {s}_j B_{\SL(4)} /B_{\SL(4)}$.
\end{enumerate}
\begin{proof}
(a): One has to show that $\chi_i([0,1]) \subseteq C_{\leq \sigma_i}(B)$ and that $\chi_i$ is a continuous map which maps $(0,1)$ homeomorphically to $C_{\sigma_i}(B)$. The first assertion is clear, since by Lemma \ref{P_i = G_iB} one has $C_{\leq \sigma_i} = G_iB/B$.

By Lemma \ref{P_i = G_iB}, one has $C_{\sigma_i}(B) = \{kB \mid k \in K_i \setminus (K_i \cap B)\}$. Let $k \in K_i \setminus (K_i \cap B)$. Then $\gamma_i^{-1}(p^{-1} (kB)) = \gamma_i^{-1}(k)\cdot B_{\SL(2,\RR)} \in \SL(2,\RR)/B_{\SL(2,\RR)} \setminus B_{\SL(2,\RR)}$. By Lemma \ref{the map R}, there exists a unique $s \in (0,1)$ satisfying $R(s)B_{\SL(2,\RR)} = \gamma_i^{-1}(k) B_{\SL(2,\RR)}$. Hence, $s$ is the unique preimage of $kB$ under $\chi_i$. This yields the desired bijectivity property. The continuity properties are clear.

(b): Since by Lemma \ref{P_i = G_iB} (b) one has $C_{\leq \sigma_i \sigma_j}(B) = K_iBK_jB/B$, it is clear that $\chi_{(i,j)}([0,1]\times [0,1]) \subseteq C_{\leq {\sigma}_i {\sigma}_j}(B)$. For the injectivity of the restriction, let $(s,t), (\tilde{s}, \tilde{t}) \in (0,1)^2$ such that $\chi_{(i,j)}(s,t) = \chi_{(i,j)}(\tilde{s}, \tilde{t})$. Then \begin{align*}
\gamma_i(R(s)) \gamma_j(R(t))B & =  \gamma_i(R(\tilde{s})) \gamma_j(R(\tilde{t}))B\\
\iff (\gamma_i(R(\tilde{s})))^{-1} \gamma_i(R(s)) \gamma_j(R(t))B & =  \gamma_j(R(\tilde{t}))B \in C_{\sigma_j}(B). \end{align*}
This implies $R(\tilde{s})^{-1} R(s) \in B_{\SO(2, \RR)}$, since otherwise the left expression is in $C_{\sigma_i\sigma_j}(B)$, contradicting $C_{\sigma_i \sigma_j}(B) \cap C_{\sigma_j}(B) = \emptyset$. Since $s, \tilde{s} \in (0,1)$, one obtains $\tilde{s} = s$. It follows that $\chi_j(t) = \chi_j(\tilde{t})$, hence $t = \tilde{t}$ by (a). 

For the surjectivity, note that by Lemma \ref{P_i = G_iB} (b), one has $C_{\sigma_i\sigma_j}(B) = Bs_is_jB/B = Bs_iBBs_jB/B$. Let $x_i x_jB$ be an arbitrary element of $C_{\sigma_i\sigma_j}(B)$ with $x_i = b_1 s_i b_2 \in B s_i B$ and $ x_j \in B s_j B$. By (a), there exists an $s \in (0,1)$ with $\gamma_i(R(s))B = b_1 s_iB \in C_{\sigma_i}(B)$. Hence, there exists a $b \in B$ with $(\gamma_i(R(s))b = b_1 s_i b_2 = x_i$. Again by (a), there exists a $t \in (0,1)$ with $\gamma_j(R(t))B = bx_jB  \in C_{\sigma_j}(B)$. This yields \begin{align*}
\chi_{i,j}(s,t) & = \gamma_i(R(s)) \cdot  \gamma_j(R(t)) B\\
& = x_i b^{-1} \cdot bx_j B\\
=x_i x_j B.
\end{align*}
This proves that $\chi_{i,j}$ maps $(0,1) \times (0,1)$ bijectively to $C_{\sigma_i\sigma_j}(B)$ The continuity properties are clear.
\end{proof}

\end{lem}

\begin{nota}
\label{epsilon}
For $i,j \in I$, let $\epsilon(i,j):= (-1)^{\alpha_j(\check{\alpha}_i)},$ where  $\alpha_j(\check{\alpha}_i)=a_{ij}$ is the $(i,j)$-entry of the Cartan matrix $\CA$ of $\Pi$. 
\end{nota}

\begin{lem}[{\cite[Remark 15.4(1)]{GHKW}}]
\label{vertauschen}
Let $e_i:= \gamma_i(-I) \in G_i$ with $\gamma_i$ as in Remark \ref{homeos alpha beta} and $k_j \in K_j$. Then $e_i k_j e_i = k_j^{\epsilon(i,j)}$.

\end{lem}

\begin{thm}
\label{fundamental group}
If the Bruhat decomposition satisfies the conclusion of Proposition \ref{bruhat decomp is cw decomp}, then a presentation of $\pi_1(G/P_J)$ is given by 

\[ \left\langle x_i; \quad i \in I \mid x_ix_j^{\epsilon(i,j) } = x_jx_i,\quad x_k = 1; \quad  i,j \in I, k \in J \right\rangle.\]
In particular, this statement holds in the $2$-spherical and the symmetrizable case.
\begin{proof}
By Lemma \ref{hom between BwB/B and BwP/P} and Proposition \ref{bruhat decomp is cw decomp}, the Bruhat decomposition $G/P_J = \bigsqcup_{w\in W^J}BwP_J/P_J$ is a CW decomposition where each cell $BwP_J/P_J$ has dimension $l(w)$. The characteristic maps of the 1-cells $Bs_iP_J/P_J$ and 2-cells $Bs_is_jP_J/P_J$ are given by the compositions $\tilde{\chi}_i := \psi_{s_i} \circ \chi_i$, respectively $\tilde{\chi}_{(i,j)}:= \psi_{s_is_j} \circ \chi_{(i,j)}$  ($\psi_{s_i}$ and $\psi_{s_is_j}$ denoting the canonical homeomorphisms from Lemma \ref{hom between BwB/B and BwP/P}). 

Lemma \ref{presentation cw} gives a presentation of $\pi_1(G/P_J)$. The generating elements are given by the homotopy classes $x_i:=[\tilde{\chi}_i]$ of the characteristic maps of the 1-cells -- namely, the cells $B s_i P_J/P_J$ where $i \in I \setminus J$. For the homotopy classes $x_k$ with $k \in J$, note that $\gamma_k(R(t)) \in G_k \subseteq P_J$, and so $\tilde{\chi}_k(t) = \gamma_k(r(t)) \cdot P_J = P_J$ which implies $x_k = [\tilde{\chi}_k] = 1_{\pi_1(G/P_J)}$. This yields the desired generating set as well as the trivial relation $x_k = 1$ for $i \in J$. %Hence, after renaming the elements of $I \setminus J$, one obtains $\{x_i \mid i \in I\setminus J\}$ as generating set. 

To obtain the set of relators, for $k = 1, \dots, 4$ let $\phi_k: [0,1] \to [0,1] \times [0,1]$ where \begin{align*}
\phi_1(t) &= (t,0),\\
\phi_2(t) &= (1,t),\\
\phi_3(t) &= (1-t,1),\\
\phi_4(t) &= (0,1-t).
\end{align*}
Then the concatenation $\phi:= \phi_1 * \phi_2 * \phi_3 * \phi_4$ is a loop in the relative boundary $\partial([0,1] \times [0,1]) \simeq \mathbb{S}^1$ which generates its fundamental group. Moreover, for each characteristic map $\tilde{\chi}_{(i,j)}$ of a 2-cell, one has $\tilde{\chi}_{(i,j)}(\phi(0)) = \tilde{\chi}_{(i,j)}((0,0)) = \psi_{s_is_j}({\chi}_{(i,j)}(0,0)) = \psi_{s_is_j}(B) = P_J$ where $P_J$ is the unique 0-cell of the CW complex. Therefore, Lemma \ref{presentation cw} implies that the set of relators is given by $\{[\tilde{\chi}_{(i,j)} \circ \phi] \mid \sigma_i\sigma_j \in W^J, l(\sigma_i\sigma_j) = 2\}$. Now, 
\begin{align*}
[\tilde{\chi}_{(i,j)} \circ \phi]  & = [\tilde{\chi}_{(i,j)} \circ \phi_1] \cdot [\tilde{\chi}_{(i,j)} \circ \phi_2] \cdot [\tilde{\chi}_{(i,j)} \circ \phi_3]\cdot [\tilde{\chi}_{(i,j)} \circ \phi_4], \end{align*}
where 

$\tilde{\chi}_{(i,j)} (s,t) = \alpha_i(R(s))\alpha_j(R(t))\cdot P_J$ with $R(0) = I_{\SO(2,\RR)}, R(1) = -I_{\SO(2,\RR)} \in B_{\SO(2,\RR)}$ which implies
\begin{align*}
[\tilde{\chi}_{(i,j)} \circ \phi_1] &=  x_i, \\
[\tilde{\chi}_{(i,j)} \circ \phi_3] &=  x_i^{-1},\\          
[\tilde{\chi}_{(i,j)} \circ \phi_4] &=  x_j^{-1}.                               
\end{align*}

Moreover, \begin{align*}
(\tilde{\chi}_{(i,j)} \circ \phi_2)(t) &= \alpha_i(-I)\alpha_j(R(t))\cdot P_J\\
& = \alpha_i(-I)\alpha_j(R(t)) \alpha_i(-I) \cdot P_J, \quad \text{ since }\alpha_i(-I)\in P_J\\
& = \alpha_j(R(t))^{\epsilon(i,j)}\cdot P_J \quad \text{ by Lemma \ref{vertauschen}}.
\end{align*}  
Since $R(t)^{-1} = R(1-t)$, this yields $[\tilde{\chi}_{(i,j)} \circ \phi_2] = x_j^{\epsilon(i,j)}$. One therefore obtains $[\tilde{\chi}_{(i,j)} \circ \phi] = x_i \cdot  x_j^{\epsilon(i,j)} \cdot x_i^{-1} \cdot x_j^{-1}$. This proves the assertion.

\end{proof}

\end{thm}

\begin{lem}
\label{cardinality fundamental group}
Let $\Pi$ be irreducible simply-laced distinct from $A_1$ and $\emptyset \neq J \subset I = \{ 1, ..., n\}$. Then $\pi_1(G/P_J) \iso {C_2}^{n-|J|}$.
\begin{proof}

For each generator $x_h$ in the presentation of Theorem~\ref{fundamental group}, one has $x_h^2 = 1$: Recall that $\lambda$ denotes the labelling map $I \to V$ of the vertex set of $\Pi$. Since $\Pi$ is connected, one has a minimal path $(i_1, \dots, i_m = h)^\lambda$ in $\Pi$ such that $i_1 \in J$. If $m = 1$, one has $x_h = 1$ by the presentation above. Let $x_{i_1}, \dots, x_{i_{m-1}}$ have order $\leq 2$. Since $\Pi$ is simply-laced, $\epsilon(m-1,h) = -1 = \epsilon (h,m-1)$ which implies $x_h  x_{i_{m-1}}^{-1} x_h^{-1} x_{i_{m-1}}^{-1} = 1$ and $x_{i_{m-1}} x_h^{-1} x_{i_{m-1}}^{-1} x_h^{-1} = 1$. Multiplying these expressions yields $x_h^2 = 1$.

Since each generator has order $\leq 2$, the relations show that the group is abelian. One concludes that $\pi_1(G/P_J) \cong {C_2}^{n-|J|}$.
%Therefore, the elements of $\pi_1(G/P_J)$ are in one-to-one correspondence with the subsets of the generator set $\{x_i \mid i = 1, \dots, n-|J|\}$. This proves the claim.
\end{proof}
\end{lem}

\section{The fundamental groups of \texorpdfstring{$G(\Pi)$}{K(Pi)} and \texorpdfstring{$\Spin(\Pi,\kappa)$}{Spin(Pi,kappa)}}
\label{spin_simply_con} \label{4}

The Iwasawa decomposition $G = KAU_+$ implies that $K$ acts transitively on the generalized flag varieties $G/P_J$. In this section, we describe the generalized flag varieties and suitable covering spaces as coset spaces of $K$ and its various spin covers defined in \cite{GHKW}. This will then allow us to compute the fundamental group of $K$ and its various spin covers via locally trivial fibre bundles and homotopy exact sequences.

\begin{lem}
\label{homeo K/(K cap P_J) to G/P_J}  \label{homeo G/P_J to K/(K cap T) K_J}
The canonical map $\psi: K/(K \cap P_J) \to G/P_J$ is a homeomorphism. In particular, there exists a homeomorphism $G/P_J \to K/(K \cap T) K_J$. %In particular, $\pi_1(K/(K \cap T) K_J) \iso {C_2}^{n-|J|}$.
\end{lem}

\begin{proof}
Bijectivity follows from the product formula for subgroups since $G = KP_J$. 
By Lemma \ref{quotient basics}, the map $\tilde{\psi}: G/(K \cap P_J) \to G/P_J$ is continuous, so the same holds for its bijective restriction $\psi:   K/(K \cap P_J) \to G/P_J$.

%Continuity:  Consider the commutative diagram 
%\begin{tikzcd}
%G \arrow[rd, "\pi"] \arrow[d, "\phi"]  \\
%G/P_J \cap K \arrow[r, "\tilde{\psi}"] & G/P_J 
%\end{tikzcd}. Proceed as above.\\

In order to show that $\psi$ is closed, let $P:= P_J$ and let $\tilde{P} := P_J \cap K$. Consider the commutative diagram \[\begin{tikzcd}
K/\tilde{P} \arrow[rd, "\psi"] \arrow[d, "\iota"]  \\
G/\tilde{P}  \arrow[r, "\phi"] & G/P 
\end{tikzcd}\] where $\iota$ denotes the canonical embedding and $\phi$ denotes the canonical map from $G/\tilde{P}$ to $G/P$. Since $K$ is closed in $G$ by \cite[Section~3F]{FHHK}, the map $\iota$ is closed. By Lemma~\ref{quotient basics}, $\phi$ is open.  

Let $X\tilde{P} \subseteq K/\tilde{P}$ be a closed subset of $K/\tilde{P}$ and 
%Then the complement $\comp_{K/\tilde{P}}(X\tilde{P})$ is open. Since $\psi$ is bijective, we have $\psi(\comp_{K/\tilde{P}}(X\tilde{P})) = \comp_{G/P}(\psi(X\tilde{P})) =  \comp_{G/P}(XP)$.  
suppose that $\psi(X\tilde{P}) = X P$ is not closed in $G/P$. Then the complement $\comp_{G/P}(XP)$ is not open in $G/P$, hence the complement $\comp_{G/\tilde{P}}(\phi^{-1} (XP)) = \phi^{-1} (\comp_{G/P}(XP))$ is not open in $G/\tilde{P}$. Therefore,  $\phi^{-1} (XP)$ is not closed in $G/\tilde{P}$. This yields that $X\tilde{P} = \psi^{-1}(XP) = \iota^{-1}(\phi^{-1}(XP))$ is not closed in $K/\tilde{P}$, a contradiction.
%{\color{red}Ist das Argument $X\tilde{P} = \psi^{-1}(XP) = \iota^{-1}(\phi^{-1}(XP))$ zulaessig? $X\tilde{P} = \psi^{-1}(XP)$ sollte gelten, da $\psi$ bijektiv. Folgt die zweite Gleichheit aus dem kommutativen Diagramm, obwohl $\iota$ nicht surjektiv ist? Falls das so nicht funktioniert, alternativer Beweis ab $(*)$:
%
%Therefore,  $\phi^{-1} (XP) = XP\tilde{P}$ is not closed in $G/\tilde{P}$. This implies that $\iota^{-1}(XP\tilde{P}) = (XP \cap K)  \tilde{P}$ is not closed in $K/\tilde{P}$. Since $X \subseteq K$, we have $XP \cap K = X(P\cap K) = X\tilde{P}$. Hence, $X\tilde{P} = (XP \cap K)  \tilde{P}$ is not closed in $K/\tilde{P}$, a contradiction.
%}  

For the second claim, since $P_J = G_JB$ and $\theta(P_J) \cap P_J = G_J T$, one has $P_J \cap K = K_J (K \cap T)$. Furthermore, $G_J$ is normal in $G_J T$ which implies $K_J(K \cap T) = (K \cap T) K_J$. The claim follows.
\end{proof}

%\begin{rem}
%Since $P_J = G_JB$ and $\omega(P_J) \cap P_J = G_J T$, we have $P_J \cap K = K_J (T \cap K)$. Furthermore, $G_J$ is normal in $G_J T$ which implies K_J(T \cap 
%
%\end{rem}
%
%\begin{cor}
%\label{homeo G/P_J to K/(K cap T) K_J}
%There exists a homeomorphism $G/P_J \to K/(K \cap T) K_J$. %In particular, $\pi_1(K/(K \cap T) K_J) \iso {C_2}^{n-|J|}$.
%\end{cor}
%
%\begin{proof}
%  Since $P_J = G_JB$ and $\theta(P_J) \cap P_J = G_J T$, one has $P_J \cap K = K_J (K \cap T)$. Furthermore, $G_J$ is normal in $G_J T$ which implies $K_J(K \cap T) = (K \cap T) K_J$. The claim now follows from Lemma~\ref{homeo K/(K cap P_J) to G/P_J}. %The second claim is immediate by Corollary~\ref{cardinality fundamental group}.
%\end{proof}

The key advantage of the description of a generalized flag variety as a $K$-coset space lies in the fact that $K \cap T$ is a finite group. It is therefore straightforward to write down covering spaces of generalized flag varieties via the following well-known basic observation from covering theory.

\begin{lem}
\label{covering basics}
Let $\phi: X \to Y$ be a continuous, open, surjective map between Hausdorff topological spaces. If all fibers are finite and of constant cardinality, then $\phi$ is a covering map.
\end{lem}
%\begin{proof}
%Let $y \in Y$ and let $\phi^{-1}(y) = \{x_1, \dots, x_k\} \subseteq X$. Since $X$ is Hausdorff, for $i = 1, \dots, k$ there exist neighborhoods $U_i$ of $x_i$ with $\bigcap_{i = 1}^k U_i = \emptyset$. Let $V:= \bigcap_{i = 1}^k \phi(U_i)$. Then $V$ is open since $\phi$ is open and $V \neq \emptyset$ since $y \in V$. The preimage $\phi^{-1}(V)$ is a disjoint union of open sets $\tilde{U}_i:= \phi^{-1}(V) \cap U_i$ and each $\tilde{U}_i$ is mapped bijectively to $V$: Surjectivity is clear; for the injectivity let $y' \in V$. Then each $\tilde{U}_i$ contains a preimage of $y'$. Since all fibers have constant cardinality $k$, it follows that $|\tilde{U}_i \cap \phi^{-1}(y')| = 1$. This proves the assertion. 
%\end{proof}

This readily applies in our setting:

\begin{lem}
\label{degree of K to K/T covering}
The canonical map $\psi: K/K_J \to K/(K \cap T)K_J$ is a covering map of degree $2^{n-|J|}$.
\end{lem}

\begin{proof}
By Lemma \ref{quotient basics}, $\psi$ is continuous, open and surjective. 

%The fibers of elements $k(K \cap T)K_J \in K/(K \cap T)K_J$ are of the form $\{kk' K_J \mid k' \in K \cap T\}$. Since by \cite[3C and Lem. 3.16]{FHHK}, the group $(K \cap T)$ has order $2^n$ and is therefore discrete, each $k' \in   K \cap T$ has an open neighborhood $U_{k'} \subseteq K$ such that $U_{k'} \cap K \cap T = k'$. Since $kU_{k'}$ is open and the projection $K \to K/K_J$ is open by Lemma \ref{quotient basics}, the set $kU_{k'} K_J \subseteq K/K_J$ is an open neighborhood of $kK_J$ which is mapped homeomorphically onto its image by $\psi$. 

%The fiber of an arbitrary element $k(K \cap T)K_J \in K/(K \cap T)K_J$ is of the form $\{kk' K_J \mid k' \in K \cap T\}$. Since by \cite[3C and Lem. 3.16]{FHHK}, the group $(K \cap T)$ has order $2^n$, its projection $(K\cap T) K_J \subseteq K/K_J$ is also finite and therefore discrete. Hence, for each $k'K_J \in  (K\cap T) K_J$ there exists an open subset $U_{k'}K_J \subseteq K/K_J$ such that $U_{k'}K_J \cap  (K\cap T) K_J  = k' K_J$. 
By \cite[Lemma~3.20 and the discussion after Prop 3.8]{FHHK}, the group $\tilde{T}:=(K \cap T)$ has order $2^n$. 
%\rot{(exakte Referenz: \cite{FHHK}, Lemma 3.20 ($T \cap K = M$) and the discussion after Prop 3.8 ($|M|=2^n$). }
%Let $\tilde{T}:= \{1= t_1 \dots t_{2^n}\} := T \cap K$ and let $k\tilde{T}K_J \in K/ \tilde{T}K_J$. We have $\psi^{-1} (k\tilde{T}K_J) = \{k t K_J\mid t \in \tilde{T}\}$. 
%For each $t \in \tilde{T}\setminus K_J$ and $k \in K$, one has $ktK_J \neq k K_J$. Since $K/K_J$ is Hausdorff and $\tilde{T}$ is finite, there exists an open neighborhood $UK_J$ of $K_J \in K/K_J$ such that for each $t \in \tilde{T}\setminus K_J$ one has $UK_J \cap UtK_J = \emptyset$. 
%
%Now, let $k\tilde{T}K_J \in K/ \tilde{T}K_J$. Since $\psi$ is open, $\psi(kUK_J) = kU \tilde{T} K_J$ is an open neighbourhood of $k \tilde{T}K_J$ in  $K/ \tilde{T}K_J$. One has $\psi^{-1}(kU \tilde{T} K_J) = \{kUt K_J\mid t \in \tilde{T}\} = \dot{\bigcup}_{t \in \tilde{T} \setminus K_J} kUt K_J$. Moreover, each open set $kUt K_J$ is mapped homeomorphically onto its image $kU \tilde{T} K_J$, since $ku_1 \tilde{T} K_J = ku_2 \tilde{T} K_J$ implies $u_1 = u_2 t K_J$ for some $t \in \tilde{T}$, and therefore $u_1 \in U \cap Ut$ which yields $t \in K_J$. 
%
%This proves that $\psi$ is a covering map.
Note that one has $T_J \cap T_{I\setminus J} = \{1\}$, since the Kac--Moody group $G$ being algebraically simply connected implies $T \iso T_J \times T_{I \setminus J}$.
%$1 \neq g \in G_J \cap G_{I\setminus J}$ implies $B \neq gB \in \Res_J(B) \cap \Res_{I \setminus J}(B) = \Res_{J \cap (I \setminus J)}(B) = \Res_\emptyset(B) = \{B\}$, a contradiction. Here, $\Res_{J}(B)$ denotes the $J$-residue of $B$ in $G/B$, that is, $\Res_{J}(B) = \{ gB \in G/B \mid g \in P_{J}\}$, where $G_{J} \subseteq G_JB = P_J$.  For the intersection of residues see, e.g., \cite[Ex. 5,32]{AB}.
Now, for $k \in K$ one has $\psi^{-1} (k\tilde{T}K_J) = \{ktK_J \mid t \in \tilde{T}\}$, and since $T_J \cap T_{I\setminus J} = \{1\}$, one has $k t_i K_J \neq  k t_j K_J$ for $t_i \neq t_j \in T \cap K_{I\setminus J}$. This yields $|\psi^{-1} (k\tilde{T}K_J)| = |\{ktK_J \mid t \in \tilde{T}\}| = |\{ktK_J \mid t \in T \cap K_{I\setminus J}\}| = | T \cap K_{I\setminus J} | = | T_{I\setminus J} \cap K_{I\setminus J} |= 2^{n-|J|}$. Lemma \ref{covering basics} now shows that $\psi$ is a covering map.\end{proof}

\label{general case}

\begin{de}[{\cite[Definition~16.2]{GHKW}}]
Let  $\Pia$ be the graph on the vertex set $V$ with edge set \[\{ \{i,j\} \in V \times V \mid i \neq j \in I, \epsilon(i,j) = \epsilon(j,i) = -1  \},\]
where $\epsilon(i,j)$ denotes the parity of the corresponding Cartan matrix entry, as defined in Notation \ref{epsilon}.

An \emph{admissible colouring} of $\Pi$ is a map $\kappa: V \to \{1,2\}$ such that
\begin{enumerate}
	\item $\kappa(i^\lambda) = 1$ whenever there exists $j \in I \setminus\{i\}$ with $\epsilon(i,j) = 1$ and $ \epsilon(j,i) = -1$.
	\item the restriction of $\kappa$ to any connected component of the graph $\Pia$ is a constant map.
\end{enumerate}
 Define $c(\Pi,\kappa)$ to be the number of connected components of $\Pia$ on which $\kappa$ takes the value 2.
For a subgraph $\Pia_J$ of $\Pia$ that is a union of connected components of $\Pia$ let $\kappa_J$ be the corresponding restriction of $\kappa$.\end{de}

\begin{de}
\label{colouring}
 Let be the colouring $\gamma: V \to \{r,g,b\}$ of $\Pia$ that to each connected component $\bar{\Pi}^{\mathrm{adm}}$ of $\Pia$ assigns a colour as follows: Let $\bar{\Pi}^{\mathrm{adm}}$ be coloured red (denoted by $r$) if it contains a vertex $i^\lambda$ such that there exists a vertex $j^\lambda \in V$ satisfying $\epsilon(i,j) = 1$ and $ \epsilon(j,i) = -1$; let $\bar{\Pi}^{\mathrm{adm}}$ be coloured green ($g$) if it is not red and consists only of an isolated vertex; and blue ($b$) else. 
%
%
%
%
%\begin{enumerate}
%	\item $\gamma(i^\lambda) = r$ whenever there exists $j \in I \setminus\{i\}$ with $\epsilon(i,j) = 1$ and $ \epsilon(j,i) = -1$.
%	\item $\gamma(i^\lambda) =g$ whenever for each $j \in I \setminus\{i\}$, one has $(\epsilon(i,j), \epsilon(j,i)) \in \{(1,1), (-1,1)\}$.
%	\item $\gamma(i^\lambda) = b$ whenever the connected component of $i^\lambda$ in $\Pia$ contains more than one vertex and case (a) applies to none of the vertices in this component.
	%there exists $j \in I \setminus \{i\}$ with $\epsilon(i,j) = -1 = \epsilon(j,i)$, there exists no $j \in I \setminus\{i\}$ with $\epsilon(i,j) = 1$ and $ \epsilon(j,i) = -1$ (i.e. case (a) does not apply), and the connected component of $i^\lambda$ in $\Pia$ does not contain a vertex $k^\lambda$ with $\gamma(k^\lambda) = r$.
%	\item The restriction of $\gamma$ to any connected component of the graph $\Pia$ is a constant map.
%\end{enumerate}
%For a vertex $i^\lambda$ of $\Pia$ with $\gamma(i^\lambda) = x \in \{r,g,b\}$, we will use the phrasing "$i^\lambda$ has colour $x$".
\end{de}
We refer to the introduction for a discussion of various examples.
%\begin{ex}
%\label{colouring example}
%\input{dynkin.tex}
%\end{ex}

\begin{de+rem}
\label{Spin remark}
Recall from the introduction that in \cite[Definition~16.16]{GHKW}, the \emph{spin group $\Spin(\Pi,\kappa)$ with respect to $\Pi$ and $\kappa$} is defined as the universal enveloping group of a particular $\Spin(2)$-amalgam $\{\tilde{G}_{ij},  \tilde{\oldphi}_{ij}^i\mid i \neq j \in I \}$ where the isomorphism type of $\tilde{G}_{ij}$ depends on the $(i,j)$- and $(j,i)$-entries of the Cartan matrix of $\Pi$ as well as the values of $\kappa$ on the corresponding vertices. The group $K(\Pi)$ can be regarded as (being uniquely isomorphic to) the universal enveloping group of an $\SO(2,\RR)$-amalgam $\{G_{ij},  \oldphi_{ij}^i\mid i \neq j \in I \}$ where each $\tilde{G}_{ij}$ covers $G_{ij}$ via an epimorphism $\alpha_{ij}$.
% the degree of the cover being $2$ if $i^\lambda$ and $j^\lambda$ are contained in a common connected component of $\Pia$ with $\kappa(i^\lambda) = 2$ and 1 else. 
By \cite[Lemma 16.18]{GHKW} there exists a canonical central extension $\rho_{\Pi,\kappa}: \Spin(\Pi, \kappa) \to K(\Pi)$ that makes the following diagram commute for all $i \neq j \in I$:  
\[\begin{tikzcd}
\tilde{G}_{ij} \arrow[r, "\tilde{\tau}_{ij}"] \arrow[d, " \alpha_{ij}"] & \Spin(\Pi,\kappa)  \arrow[d, "\rho_{\Pi,\kappa}"] \\
G_{ij} \arrow[r, "{\tau}_{ij}"] & K(\Pi) 
\end{tikzcd}\] 
Here, $\tilde{\tau}_{ij}$ and ${\tau}_{ij}$ denote the respective canonical maps into the universal enveloping groups. 

By \cite[Proposition~3.9]{GHKW}, one has \[\ker(\rho_{\Pi,\kappa}) = \langle \tilde{\tau}_{ij}(\ker(\alpha_{ij})) \mid i \neq j \in I \rangle_{\Spin(\Pi,\kappa)}.\]

 Each connected component of $\Pia$ that admits a vertex $i^\lambda$ with $\kappa(i^\lambda) = 2$ contributes a factor $2$ to the order of $\ker(\rho_{\Pi,\kappa})$ so that $\Spin(\Pi,\kappa)$ is a $2^{c(\Pi,\kappa)}$-fold central extension of $K(\Pi)$. 

In particular, this implies that the subspace topology on $K(\Pi)$ defines a unique topology on $\Spin(\Pi)$  that turns the extension into a covering map. The resulting group topology on $\Spin(\Pi, \kappa)$ is called the {\em Kac--Peterson topology} on $\Spin(\Pi, \kappa)$. 

In the case of an irreducible simply-laced diagram $\Pi$, the only admissible colourings are the (trivial) constant colouring $V \to \{1\}$ (that every diagram admits) and the constant colouring $\kappa: V \to \{2\}$; we define the \emph{spin group $\Spin(\Pi)$ with respect to $\Pi$} as $\Spin(\Pi) := \Spin(\Pi, \kappa)$.

\end{de+rem}

Before turning to the general case, we will first consider the simply-laced case and formulate and prove the corresponding simplified versions of the main theorems.

%\rot{Spaetestens an dieser Stelle muesste $\Spin(\Pi)$ fuer $\Pi$ simply laced definiert werden: \cite[Definition 11.5]{GHKW}}

\begin{lem}
\label{homeo Spin(Pi)/Spin(Pi_{ij}) and K/K_{ij}}
Let $\Pi$ be irreducible simply-laced distinct from $A_1$ and let $\{i,j\} \subseteq I$ be the index set of an $A_2$-subdiagram of $\Pi$. Then the spaces $\Spin(\Pi)/\Spin(\Pi_{ij})$ and $K/K_{ij}$ are homeomorphic.

\begin{proof}
From \cite{GHKW} (exact references below) it follows that the kernel of the covering map $\Spin(\Pi) \to K$ coincides with the kernel of the covering map $\Spin(\Pi_{ij}) \to K_{ij}$ and is equal to the group $Z:= \{\pm 1_{\Spin(\Pi)}\}$ (for the definition of $-1_{\Spin(\Pi)}$, see below). 
%In \cite{GHKW} it is established that for any irreducible simply-laced Dynkin diagram $\Pi$ the covering map $\Spin(\Pi) \to K(\Pi)$ has kernel $Z:= \{\pm 1_{\Spin(\Pi)}\}$. 
This is a consequence of the following facts regarding an irreducible simply-laced diagram $\Pi$ (all referring to \cite{GHKW}):
\begin{itemize}
\item There is an epimorphism $\Spin(2) \to \SO(2,\RR)$ with kernel $\{\pm 1_{\Spin(2)}\}$ (see [Theorem~6.8]).
\item In $\Spin(\Pi)$, all elements $\tilde{\tau}_{ij}(\tilde{\oldphi}_{ij}^i(-1_{\Spin(2)}))$ coincide
%, i.e., the $-1$-elements of the $\Spin(2)$-subgroups of $\Spin (\Pi)$ corresponding to the amalgam $\mathcal{A}(\Pi, \Spin(2))$ 
(see [Lemma~11.7]).
\item Let $-1_{\Spin(\Pi)}:= \tilde{\tau}_{ij}(\tilde{\oldphi}_{ij}^i(-1_{\Spin(2)}))$ for an arbitrary pair $i \neq j \in I$.
% $\Spin(2)$-subgroup $\lambda_{ij}(\tilde{\phi}_{ij}^i(\Spin(2)))$. 
Then $1_{\Spin(\Pi)} \neq -1_{\Spin(\Pi)}$ (see [Corollary~11.16]).
\item $\Spin(\Pi)$ is a 2-fold central extension of $K(\Pi)$ (see [Theorem~11.17]).
\end{itemize}
Hence, the 2-fold covering map $\tilde{\phi}: \Spin (\Pi) \to K(\Pi)$ induces a continuous bijective map $\phi: \Spin(\Pi) / \Spin(\Pi_{ij}) \to  (\Spin(\Pi) / Z) / (\Spin(\Pi_{ij}) / Z) \to K / K_{ij}$. One has a commutative diagram \[\begin{tikzcd}
\Spin(\Pi) \arrow[r, "\tilde{\phi}"] \arrow[d, "\pi_1"] & K \arrow[d, "\pi_2"] \\
\Spin(\Pi)/\Spin(\Pi_{ij}) \arrow[r, "\phi"] & K/K_{ij} 
\end{tikzcd}.\]
Since $\tilde{\phi}$ is open as a covering map and $\pi_2$ is open by Lemma \ref{quotient basics}, it follows that $\phi$ is a homeomorphism.  
\end{proof}
\end{lem}

\begin{lem}
\label{K/K_J is simply connected.}
Let $\Pi$ be irreducible simply-laced distinct from $A_1$ and $\emptyset \neq J \subset I = \{ 1, ..., n\}$. Then $K/K_J$ is simply connected.
\end{lem}

\begin{proof}
$K/K_J$ is connected since $K$ is generated by connected groups isomorphic to $\SO(2,\RR)$. Hence by Lemma \ref {degree of K to K/T covering} it is a non-trivial cover of $K/(K \cap T)K_J$ of degree $2^{n-|J|}$. The claim now follows from Corollary~\ref{cardinality fundamental group} and Corollary~\ref{homeo G/P_J to K/(K cap T) K_J}.
\end{proof}

%Recall from the introduction that $\Spin(\Pi)$ is the universal enveloping group of the $\Spin(2)$-amalgam $\mathcal{A}(\Pi, \Spin(2)) =  \{\hat{H}_{ij}, \hat{\oldphi}_{ij}^i \mid i \neq j \in I\}$. For each pair $i \neq j$ in $I$ there exists a canonical homomorphism $\hat{\tau}_{ij}: \hat{H}_{ij} \to \Spin(\Pi)$ where $\hat{H}_{ij}$ is equal to either $\Spin(3)$ or $\Spin(2) \times \Spin(2)$, depending on whether $(i,j)^\lambda$ is an edge of $\Pi$ or not. Let $\Spin(\Pi_{ij}):= \tau_{ij}(\hat{H}_{ij})$. 

%By \cite[4.2.4]{Hus}, for a closed subgroup $H$ of a topological group $G$, the projection $p: G \to G/B$ is a principal $G$-bundle. By \cite[4.1, Cor.]{Pal}, this bundle is locally trivial if $H$ is a closed Lie group. Since locally trivial bundles admit local cross sections, \cite[7.4, Cor.]{Ste} implies that in this case $p: G \to G/B$ is fibre bundle with fibre $H$. This proves the following Lemma.
%
%
%\begin{lem}
%Let $J$ be as above. We have a fibre bundle 
%\[\begin{tikzcd}
%\Spin(\Pi_J) \arrow[r] & \Spin(\Pi) \arrow[r] & \Spin(\Pi)/\Spin(\Pi_J). 
%\end{tikzcd}\]
%\end{lem}

The following proposition provides our main result in the simply laced case.

\begin{prop} \label{simplyconnected}
Let $\Pi$ be irreducible simply-laced distinct from $A_1$. Then $\Spin(\Pi)$ is simply connected with respect to the Kac--Peterson topology. In particular, $\pi_1(G) \iso C_2$.

\begin{proof}
By \cite[4.2.4]{Hus}, for a closed subgroup $H$ of a topological group $G$, the projection $p: G \to G/H$ is a principal $H$-bundle. By Lemma \ref{Palais}, %\cite[Corollary in Section~4.1]{Pal},
this bundle is locally trivial if $H$ is a (closed) Lie group (note that, by \cite[Theorem~5.11]{HR}, every locally compact subgroup of a topological group is closed). Since locally trivial bundles admit local cross sections, \cite[Corollary in Section~7.4]{Ste1} implies that, if $H$ is a closed Lie group, then $p: G \to G/H$ is a fibre bundle with fibre $H$. This yields a {locally trivial} fibre bundle
\[\begin{tikzcd}
\Spin(\Pi_{ij}) \arrow[r] & \Spin(\Pi) \arrow[r] & \Spin(\Pi)/\Spin(\Pi_{ij}). 
\end{tikzcd}\]
By \cite[Chapter~4]{Hatcher}, this yields the homotopy long exact sequence
\begin{eqnarray}
\pi_4(\Spin(\Pi)/\Spin(\Pi_{ij})) & \to & \pi_3(\Spin(\Pi_{ij})) \to \pi_3(\Spin(\Pi)) \to
\pi_3(\Spin(\Pi)/\Spin(\Pi_{ij})) \notag \\ & \to & \pi_2(\Spin(\Pi_{ij})) \to \pi_2(\Spin(\Pi)) \to
\pi_2(\Spin(\Pi)/\Spin(\Pi_{ij})) \notag \\ & \to & \pi_1(\Spin(\Pi_{ij})) \rightarrow \pi_1(\Spin(\Pi)) \rightarrow \pi_1(\Spin(\Pi)/\Spin(\Pi_{ij})) \notag \\  \label{homotopylong}
\end{eqnarray}
from which one extracts the exact sequence
\[ \{1\} = \pi_1(\Spin(\Pi_{ij})) \rightarrow \pi_1(\Spin(\Pi)) \rightarrow \pi_1(\Spin(\Pi)/\Spin(\Pi_{ij})). \]
By Lemmas~~\ref{homeo Spin(Pi)/Spin(Pi_{ij}) and K/K_{ij}} and \ref{K/K_J is simply connected.} one has $ \pi_1(\Spin(\Pi)/\Spin(\Pi_{ij})) \iso \pi_1(K/K_{ij}) = \{1\}$ and so by exactness $\pi_1(\Spin(\Pi)) = \{ 1 \}$.

The second assertion follows from the fact that $\pi_1(G) \iso \pi_1(K)$ by Corollary \ref{FundamentalGroups2} and the fact that $\Spin(\Pi)$ is a 2-fold central extension of $K$ by \cite[Theorem~11.17]{GHKW}.
\end{proof}
\end{prop}

We will now return to the case of a general irreducible Dynkin diagram $\Pi$.

\begin{nota}
\label{HJ}
For a subset $J \subseteq I$ let
\[ H_J:= \left\langle x_i; \quad i \in J \mid x_ix_j^{\epsilon(i,j) } = x_jx_i,; \quad  i,j \in J\right\rangle.\]
\end{nota}

\begin{lem}
\label{connected component subgroups}
Let $J \subseteq I$ be the index set of a connected component $\Pia_J$ of $\Pia$. Then the following hold:
\begin{enumerate}
	\item If $\Pia_J$ has colour $r$, then $H_J \iso C_2^{|J|}$.
	\item If $\Pia_J$ has colour $g$, then $|J| = 1$ and $H_J \iso \ZZ$.
	\item If $\Pia_J$ has colour $b$, then $|H_J| = 2^{|J|+1}$.
\end{enumerate}
\end{lem}

\begin{proof}
(a): If $\Pia_J$ has colour $r$, then there exist $i \in J, j \in I \setminus\{i\}$ with $\epsilon(i,j) =1$ and $\epsilon(j,i) = -1$. This implies $x_ix_j = x_j x_i$ and $x_j x_i^{-1} = x_i x_j$ which yields $x_i^2 = 1$. Now, if $\{i^\lambda, k^\lambda\}$ is an edge in $\Pia$, then $x_i x_k^{-1}x_i^{-1} x_k^{-1} = 1 = x_k x_i^{-1} x_k^{-1} x_i^{-1}$. Multiplying these expressions shows that $x_i^2 = 1$ implies $x_k^2 = 1$. Since $\Pia_J$ is connected, this yields $x_k^2 = 1$ for each $k \in J$. Commutativity then follows from the relations of $H_J$. 

(b): By definition, nodes of colour $g$ are isolated in $\Pia$.

(c): Let $\Pi_J^{\mathrm{sl}}$ be the simply laced Dynkin diagram with vertex set $J^\lambda$ and edge set $\{ \{i,j\} \in J \times J \mid \{i,j\} \text{ edge in }\Pia\}$.
Let $\tilde{T}:= K(\Pi_J^{\mathrm{sl}}) \cap T(\Pi_J^{\mathrm{sl}})$ where $T(\Pi_J^{\mathrm{sl}})$ denotes the standard maximal torus of $G(\Pi_J^{\mathrm{sl}})$. Then by Lemma \ref{degree of K to K/T covering} and Proposition \ref{simplyconnected}, $\Spin(\Pi_J^{\mathrm{sl}}) \to  K(\Pi_J^{\mathrm{sl}})  \to K(\Pi_J^{\mathrm{sl}})/\tilde{T}$ is a universal covering map where $K(\Pi_J^{\mathrm{sl}})  \to K(\Pi_J^{\mathrm{sl}})/\tilde{T}$ has degree $2^{|J|}$ and  $\Spin(\Pi_J^{\mathrm{sl}}) \to  K(\Pi_J^{\mathrm{sl}})$ has degree $2$ according to  \cite[Theorem~11.17]{GHKW}. Since  $\pi_1(K(\Pi_J^{\mathrm{sl}})/\tilde{T}) \iso H_J$ by Theorem \ref{fundamental group} and Lemma~\ref{homeo K/(K cap P_J) to G/P_J}, this implies $|H_J| = 2^{|J|+1}$.
%Then by the results for the simply laced case,
%
%\rot{Give concrete references:
%
%Then by Theorem 3.6
  %$\Spin() \to K() \to K(...)/\tilde{T}$
%is a universal covering map.
%According to Proposition 2.13 one has
  %$\pi_1(K(...)/ \tilde T) \cong H_J$.}
%
%
 %$\Spin(\Pi_J^{\mathrm{sl}})$ is the universal cover of $K(\Pi_J^{\mathrm{sl}})/\tilde{T}$ where $\pi_1(K(\Pi_J^{\mathrm{sl}})/\tilde{T}) \iso H_J$; 
%furthermore, $K(\Pi_J^{\mathrm{sl}}) \to  K(\Pi_J^{\mathrm{sl}})/\tilde{T}$ is a covering map of degree $2^{|J|}$ 
%and by \cite[Theorem~11.17]{GHKW}, $\Spin(\Pi_J^{\mathrm{sl}})$ is a double cover of $K(\Pi_J^{\mathrm{sl}})$. This implies $|H_J| = |\pi_1(K(\Pi_J^{\mathrm{sl}})/\tilde{T}| = 2^{|J|+1}$.
\end{proof}

\begin{prop}
\label{fundamental group as product}
Let $J_1 \sqcup \dots \sqcup J_k = I$ be the index sets of the connected components of $\Pia$. If the Bruhat decomposition satisfies the conclusion of Proposition \ref{bruhat decomp is cw decomp}, then \[\pi_1(G/B) \iso H_{J_1} \times \dots \times H_{J_k}. \] 

\begin{proof}
By Theorem \ref{fundamental group}, $\pi_1(G/B) \iso H_I$ where \[H_I = \left\langle x_i; \quad i \in I \mid x_ix_j^{\epsilon(i,j) } = x_jx_i,; \quad  i,j \in I\right\rangle\] as defined in \ref{HJ}. For $J \subseteq I$, let 
\begin{equation}
\label{R_J}
R_J:= \{x_ix_j^{\epsilon(i,j)}x_i^{-1} x_j^{-1} \mid i,j \in J\},
\end{equation} the set of relators of $H_J$. Let \[R^c:= \bigcup_{\substack{i^\lambda, j^\lambda \text{in different}\\\text{conn. components}}}\{x_ix_j x_i^{-1} x_j^{-1}\},\] the set of commutators of pairs of generators from different connected components of $\Pia$. Then \[H_{J_1} \times \dots \times H_{J_k} \iso \left \langle x_i; \quad i \in I \mid \bigcup_{l=1}^k R_{J_l} \cup R^c \right\rangle =: H.\] Let $\pi_{H_I}$ and $\pi_{H}$ be the canonical homomorphisms from the free group $\langle x_i; \quad i \in I\rangle$ to $H_I$ and $H$, respectively. It suffices to show that $\bigcup_{l=1}^k R_{J_l} \cup R^c \subseteq \ker \pi_{H_I}$ and $R_{I} \subseteq \ker \pi_H$.
It is clear that a relator $x_ix_j^{\epsilon(i,j)}x_i^{-1} x_j^{-1} \in R_I$ with $i^\lambda$ and $j^\lambda$ in a common connected component is contained in $\bigcup_{l=1}^k R_{J_l} \subseteq \ker \pi_H$, so let $x_ix_j^{\epsilon(i,j)}x_i^{-1} x_j^{-1} \in R_I$ with $i^\lambda$ and $j^\lambda$ in different connected components of $\Pia$. Then one has $(\epsilon(i,j), \epsilon(j,i)) \in \{ (1,1), (1,-1), (-1,1)\}$. If $\epsilon(i,j) = 1$, then $x_ix_j^{\epsilon(i,j)}x_i^{-1} x_j^{-1} \in R^c \subseteq \ker \pi_H$, so let $\epsilon(i,j) = -1$ and $\epsilon (j,i) = 1$. Then $j^\lambda$ is contained in a connected component $\Pia_{J_m}$ of colour $r$, and by Lemma \ref{connected component subgroups}, $\langle x_l; \quad l \in J_m \mid R_{J_m} \rangle = H_{J_m} \iso C_2^{|J_m|}$.

This implies that $x_j$ has order $2$ in ${H_{J_m}}$, hence $x_j^2 \in \langle \langle R_{J_m} \rangle \rangle_{\langle x_i;  i \in I\rangle}$, the normal closure of $R_{J_m}$ in the free group.

 Since $\langle \langle R_{J_m} \rangle \rangle_{\langle x_i;  i \in I\rangle} \subseteq \ker \pi_H$, one obtains $x_j^2 \in \ker \pi_H$. Since $x_j x_i x_j^{-1} x_i^{-1} \in R^c \subseteq \ker \pi_H$ and $\epsilon(i,j) = -1$, one therefore has \[\pi_H(x_ix_j^{\epsilon(i,j)}x_i^{-1} x_j^{-1}) = \pi_H (x_j x_i x_j^{-1} x_i^{-1} \cdot x_ix_j^{\epsilon(i,j)}x_i^{-1} x_j^{-1} ) = 1_H.\] 
%using the fact that $x_j^{-1}x_j^{\epsilon(i,j)} = x_j^{-2} \in \ker \pi_H.$ 

Conversely, it is clear that $\bigcup_{l=1}^k R_{J_l}  \subseteq R_{I} \subseteq \ker \pi_{H_I}$, so let $x_i x_j x_i^{-1} x_j^{-1} \in R^c$ with $i^\lambda$ and $j^\lambda$ in different connected components. As above, we can assume that $\epsilon(i,j) = -1$ and $\epsilon(j,i) = 1$. Since $x_j x_i^{\epsilon(j,i)} x_j^{-1} x_i^{-1} \in \ker \pi_{H_I}$, this implies  \[ \pi_{H_I}(x_i x_j x_i^{-1} x_j^{-1}) = \pi_{H_I}(x_j x_i^{\epsilon(j,i)} x_j^{-1} x_i^{-1} \cdot x_i x_j x_i^{-1} x_j^{-1}) = 1_{H_I}.\] This proves the assertion.
\end{proof}
\end{prop}

%\begin{lem}[{\cite[Theorem~1.1]{Calcut}}]
%\label{surjectivity of quotient-induced homo}
%Let $f:(X,x_0)\to(Y,y_0)$ be a quotient map of topological spaces, where $X$ is locally path-connected and $Y$ is semilocally simply-connected. If each fiber $f^{-1}(y)$ is connected, then the induced homomorphism \mbox{$f_{*}:\pi_1(X,x_0)\to \pi_1(Y,y_0)$} is surjective.
%\end{lem}
%
%\begin{cor}
%Let $J \subseteq I$ and $\phi: K \to K/K_J$ the canonical projection. Then the induced homomorphism $\phi_*:\pi_1(K) \to \pi_1(K/K_J)$ is surjective.
%
%\begin{proof}
%Covering spaces of CW complexes carry an induced CW structure -- the characteristic maps lift to characteristic maps of the covering space; see e.g. \cite[Theorem~8.10]{Bredon}. 
%Lemma \ref{homeo K/(K cap P_J) to G/P_J} and Lemma \ref{degree of K to K/T covering} therefore imply that $K$ and $K/K_J$ are CW complexes. By \cite[Proposition~A.4]{Hatcher}, this implies that they are locally contractible, and hence locally path-connected as well as semilocally simply connected. Since the fibers of $\phi$ are translates of $K_J$ and therefore connected, Lemma \ref{surjectivity of quotient-induced homo} shows that $\phi_*$ is surjective.
%\end{proof}
%\end{cor}

\begin{thm}
\label{fundamental group of K}
Let $\Pi$ be an irreducible Dynkin diagram such that $G(\Pi)$ satisfies the conclusions of Proposition \ref{bruhat decomp is cw decomp} and of Theorem \ref{FundamentalGroups2}. Let $n(g)$ and $n(b)$ be the number of connected components of $\Pia$ of colour $g$ and $b$, respectively. Then 
\[\pi_1(G(\Pi)) \iso \ZZ^{n(g)} \times C_2^{n(b)}.\] In particular, this statement holds in the symmetrizable case.
%\rot{Fuer $G(\Pi)$ statt $K(\Pi)$ formulieren und auf Anhang verweisen.}
\end{thm}

\begin{proof}
By Theorem \ref{FundamentalGroups2}, $\pi_1(G) \iso \pi_1(K)$, so it suffices to prove that $\pi_1(K)$ is of the given isomorphism type; note that Theorem \ref{FundamentalGroups2} has only been established in the symmetrizable case.
Let $J \subseteq I$. The diagram \[\begin{tikzcd}
K \arrow[r, "\phi"] \arrow[d, "p"] & K/K_J \arrow[d, "q"] \\
K/(K \cap T)\arrow[r, "\psi"] & K/(K \cap T)K_J \end{tikzcd},\] with all maps being the respective canonical maps, commutes. Since the maps are continuous by Lemma \ref{quotient basics}, one obtains a commutative diagram of induced homomorphisms \[\begin{tikzcd}
\pi_1(K) \arrow[r, "\phi_*"] \arrow[d, "p_*"] & \pi_1(K/K_J) \arrow[d, "q_*"] \\
\pi_1(K/(K \cap T))\arrow[r, "\psi_*"]& \pi_1(K/(K \cap T)K_J) \end{tikzcd},\] 
where $p_*$ and $q_*$ are injective, because $p$ and $q$ are covering maps (see Lemma~\ref{degree of K to K/T covering}). By Theorem \ref{fundamental group} and Lemma \ref{homeo K/(K cap P_J) to G/P_J}, $\pi_1(K/(K \cap T))$ and  $\pi_1(K/(K \cap T)K_J)$ can be identified with $H_I = \langle x_i; \quad i \in I \mid R_I\rangle$ and $\langle x_i; \quad i \in I \mid R_I \cup \{x_j \mid j \in J\}\rangle$, respectively ($R_I$ as in (\ref{R_J}) in the above proof), where $\psi_*$ corresponds to the canonical homomorphism between these groups as the proof of Theorem \ref{fundamental group} shows.

For the index set $J_m$ of a connected component of $\Pia$, let $\bar{J}_m:= I \setminus J_m$. Then by Proposition \ref{fundamental group as product},  \[\left\langle x_i; \quad i \in I \mid R_I \cup \{x_j \mid j \in \bar{J}_m\}\right \rangle \iso \left(\prod_{i=1}^k{H_{J_i}}\middle/\prod_{\substack{i=1\\i \neq m}}^k{H_{J_i}}\right) \iso H_{J_m}.\] Summing up, one obtains a commutative diagram \[\begin{tikzcd}
\pi_1(K) \arrow[r, "\phi_*"] \arrow[d, "p_*"] & \pi_1(K/K_{\bar{J}_m}) \arrow[d, "q_*"] \\
\prod_{i=1}^k{H_{J_i}} \arrow[r, "\pi_m"]& H_{J_m} \end{tikzcd},\] having replaced $p_*$ and $q_*$ from above with the corresponding monomorphisms.

By Lemma \ref{degree of K to K/T covering}, the covering $K/K_{\bar{J}_m} \to K/K_{\bar{J}_m}(K \cap T)$ has degree $2^{n-|\bar{J}_m|} = 2^{|J_m|}.$ This implies that $\tilde{H}_m:= q_*(\pi_1(K/K_{\bar{J}_m}))$ is a subgroup of $H_{J_m}$ of index $2^{|J_m|}$. The isomorphism type of $\tilde{H}_m$ is uniquely determined by this index and Lemma \ref{connected component subgroups}: One has \[\tilde{H}_m \iso \begin{cases} \{1\}, & \text{if }\Pia_{J_m}\text{ has colour }r,\\
 2\ZZ \iso \ZZ, & \text{if }\Pia_{J_m}\text{ has colour }g,\\
C_2 & \text{if }\Pia_{J_m}\text{ has colour }b. \end{cases}\] 
Again by Lemma \ref{degree of K to K/T covering}, the covering $K \to K/(K \cap T)$ has degree $2^n$, so $p_*(\pi_1(K))$ is a subgroup of index $2^n$ of $\prod_{i=1}^k{H_{J_i}}$. The commutative diagram above implies that $\pi_1(K) \iso p_*(\pi_1(K)) \subseteq \pi_m^{-1}(\tilde{H}_m)$. Since this holds for the index set of every connected component of $\Pia$, one has $p_*(\pi_1(K)) \subseteq \tilde{H}_1 \times \dots \times \tilde{H}_m$. But the latter is a subgroup of index $2^{|J_1|}\cdot \dots \cdot 2^{|J_m|} = 2^n$ of $\prod_{i=1}^k{H_{J_i}}$, so equality holds. This proves the assertion.
\end{proof}

%\rot{Spaetestens an dieser Stelle muessten $\Spin(\Pi, \kappa)$ und $ \rho_{\Pi, \kappa}$ definiert werden: \cite[Definition 16.16, Lemma 16.18]{GHKW}.}

%\begin{ex}
%For the Dynkin diagrams from example \ref{colouring example} one obtains
%\begin{itemize}
	%\item 
%\end{itemize}
%\rot{Give $\pi_1(G)$ for the groups from example \ref{colouring example}}
%\end{ex}

%\begin{rem}
%\label{Spin remark}
%Recall from the introduction that $\Spin(\Pi, \kappa)$ is the universal enveloping group of the $\Spin(2)$-amalgam $\mathcal{A}(\Pi, \Spin(2)) =  \{\tilde{H}_{ij}, \hat{\oldphi}_{ij}^i \mid i \neq j \in I\}$, that $K(\Pi)$ can be regarded as the universal enveloping group of an $\SO(2)$-amalgam $\{G_{ij},  \oldphi_{ij}^i\mid i \neq j \in I \}$ where each $\tilde{G}_{ij}$ covers $G_{ij}$ via an epimorphism $\alpha_{ij}$, and that there exists a canonical central extension $\rho_{\Pi,\kappa}: \Spin(\Pi, \kappa) \to K(\Pi)$ that makes the following diagram commute for all $i \neq j \in I$:  
%\[\begin{tikzcd}
%\tilde{G}_{ij} \arrow[r, "\tilde{\tau}_{ij}"] \arrow[d, " \alpha_{ij}"] & \Spin(\Pi,\kappa)  \arrow[d, "\rho_{\Pi,\kappa}"] \\
%G_{ij} \arrow[r, "{\tau}_{ij}"] & K(\Pi) 
%\end{tikzcd}\] 
%Here, $\tilde{\tau}_{ij}$ and ${\tau}_{ij}$ denote the respective canonical maps into the universal enveloping groups. 
%\end{rem}

\begin{thm}
\label{fundamental group of Spin}
Let $\Pi$ be an irreducible Dynkin diagram such that $G(\Pi)$ satisfies the conclusion of Proposition \ref{bruhat decomp is cw decomp}. Let $n(g)$ be the number of connected components of $\Pia$ of colour $g$. Let $n(b,\kappa)$ be the number of connected components of $\Pia$ on which $\kappa$ takes the value 1 and which have colour $b$. Then 
\[\pi_1(\Spin(\Pi,\kappa)) \iso \ZZ^{n(g)} \times C_2^{n(b,\kappa)}.\] In particular, this statement holds in the $2$-spherical and the symmetrizable case.
\end{thm}

\begin{proof}
By \cite[Theorem~17.1]{GHKW}, the map $ \rho_{\Pi,\kappa}:\Spin(\Pi, \kappa) \to K$ is a $2^{c(\Pi, \kappa)}$-fold central extension. Let $J$ be the index set of a connected component of $\Pia$ and let $\bar{J}:= I \setminus \bar{J}$. Let $U_{\bar{J}} := \langle \tilde{G}_{ij} \mid i \neq j \in J \rangle_{\Spin(\Pi,\kappa)}$.

 Since $\rho_{\Pi,\kappa}(U_{\bar{J}}) \subseteq K_{\bar{J}}$, one has a continuous induced map $\rho_{\Pi,\kappa}^J: \Spin(\Pi,\kappa) / \Spin(\Pi_{\bar{J}},\kappa_{\bar{J}}) \to K/K_{\bar{J}}$ making the following diagram commute, where $\tilde{\phi}$ and $\phi$ denote the respective canonical maps:
\[\begin{tikzcd}
\Spin(\Pi, \kappa) \arrow[r, "\tilde{\phi}"] \arrow[d, " \rho_{\Pi,\kappa}"] & \Spin(\Pi,\kappa) /U_{\bar{J}} \arrow[d, "\rho_{\Pi,\kappa}^J"] \\
K \arrow[r, "\phi"] & K/K_{\bar{J}} 
\end{tikzcd}\]

Each fiber of $\rho_{\Pi,\kappa}^J$ has cardinality \begin{align*}
|\{xU_{\bar{J}} \mid x \in \ker \rho_{\Pi,\kappa}\}| &= |\ker(\rho_{\Pi,\kappa}) /(U_{\bar{J}} \cap \ker(\rho_{\Pi,\kappa}))|\\
%& = |\ker \rho_{\Pi,\kappa} / \ker \rho_{\Pi_{\bar{J}},\kappa_{\bar{J}}}| \\
& = 2^{c(\Pi, \kappa) - c(\Pi_{\bar{J}},\kappa_{\bar{J}}) }\text{ by Remark \ref{Spin remark}}
\end{align*} 
Since $\rho_{\Pi,\kappa}$ is open as a covering map and $\phi$ is open by Lemma \ref{quotient basics}, it follows from Lemma \ref{covering basics} that $\rho_{\Pi,\kappa}^J$ is a covering map.

From here the proof is analogous to the proof of Theorem \ref{fundamental group of K}, after extending the commutative diagram at the beginning of the latter proof:
\[\begin{tikzcd}
\Spin(\Pi, \kappa) \arrow[r, "\tilde{\phi}"] \arrow[d, " \rho_{\Pi,\kappa}"] & \Spin(\Pi,\kappa) / U_{\bar{J}} \arrow[d, "\rho_{\Pi,\kappa}^J"] \\
K \arrow[r, "\phi"] \arrow[d, "p"] & K/K_J \arrow[d, "q"] \\
K/(K \cap T)\arrow[r, "\psi"] & K/(K \cap T)K_J \end{tikzcd}.\]
One obtains that $\pi_1(\Spin(\Pi, \kappa)) \iso \prod_{i = 1}^k H'_{J_i}$ where  each $H'_{J_m}$ is a subgroup of index $2^{c(\Pi, \kappa) - c(\Pi_{\bar{J}_m},\kappa_{\bar{J}_m})}$ of \[\tilde{H}_m \iso \begin{cases} \{1\}, & \text{if }\Pia_{J_m}\text{ has colour }r,\\
 2\ZZ \iso \ZZ, & \text{if }\Pia_{J_m}\text{ has colour }g,\\
C_2 & \text{if }\Pia_{J_m}\text{ has colour }b. \end{cases}\] Since $\Pia_{\bar{J}_m}$ is the union of all connected components except $\Pia_{J_m}$, one has $c(\Pi, \kappa) - c(\Pi_{\bar{J}_m},\kappa_{\bar{J}_m}) \in \{0,1\}$, depending on whether $\kappa$ is constant $1$ or $2$ on  $\Pia_{J_m}$. This implies \[{H}'_m \iso \begin{cases} \{1\}, & \text{if }\Pia_{J_m}\text{ has colour }r,\\
 \ZZ, & \text{if }\Pia_{J_m}\text{ has colour }g,\\
C_2 & \text{if }\Pia_{J_m}\text{ has colour }b\text{ and }\kappa\equiv 1\text{ on }\Pia_{J_m},\\
\{1\}, & \text{if }\Pia_{J_m}\text{ has colour }b\text{ and }\kappa\equiv 2\text{ on }\Pia_{J_m}.
 \end{cases}\]
This proves the assertion. 
\end{proof}

Now all theorems from the introduction have been proved.

\appendix

\section{Maximal unipotent subgroups of Kac--Moody groups and applications to Kac--Moody symmetric spaces (by Tobias Hartnick and Ralf K\"ohl)}

Throughout this appendix we fix a symmetrizable generalized Cartan matrix ${\bf A}$ with underlying diagram $\Pi$. We consider the corresponding algebraically simply-connected semisimple split real Kac--Moody group $G:=G(\Pi) = [G_\CA(\RR),G_\CA(\RR)]$ as given by Definition \ref{Kac--Moody group}. As in Section \ref{Split-real Kac--Moody groups} we also denote by $K_\CA(\RR) \leq G_\CA(\RR)$ the fixed point subgroup of the Cartan--Chevalley involution $\theta$ and set $K := K(\Pi) = K_\CA(\RR) \cap G$. We equip all of these groups with the restrictions of the Kac--Peterson topology.

The goal of this appendix is to relate the topology of $G$ to the topology of $K$. Our main result (see Theorem \ref{FundamentalGroups2} below) asserts that the inclusion $K \hookrightarrow G$ is a weak homotopy equivalence. This implies in particular that $\pi_1(G) \cong \pi_1(K)$ and thus allows the computation of $\pi_1(G)$ by the methods presented in the main part of the article.

In the spherical case the subgroup $K < G$ is even a deformation retract and hence the inclusion $K \hookrightarrow G$ is a homotopy equivalence, as a consequence of the topological Iwasawa decomposition of $G$. This decomposition also implies that the associated Riemannian symmetric space $G/K$ is contractible.

While real Kac--Moody groups also possess an Iwasawa decomposition, it is currently unknown whether this decomposition is topological. To establish our main result we thus have to work with a certain central quotient $\overline{G}$ of $G$, for which the topological Iwasawa decomposition was established in \cite{FHHK}. We will show that the image $\overline{K}$ of $K$ in $\overline{G}$ is a strong deformation retract and that the reduced Kac--Moody symmetric space $\overline{G}/\overline{K}$ is contractible. Since the finite-dimensional central extension $G \to \overline{G}$ is a Serre fibration by a classical result of Palais \cite{Pal}, this will allow us to deduce the desired result about $G$ and $K$.

\subsection{The topological Iwasawa decomposition}
Let us denote by $\Ad: \calG_\CA(\RR) \to \Aut(\g_\RR(\CA))$ and $\Ad: G(\Pi) \to \Aut(\g_\RR'(\CA))$ the adjoint representations of $ \calG_\CA(\RR)$ and $G = G(\Pi)$ respectively. We recall from \cite{FHHK} that the quotient map $G \to {\rm Ad}(G)$ factors as
\begin{equation}\label{AdjointSemisimpleQuotient}
G \xrightarrow{p_1} \overline{G} \xrightarrow{p_2} {\rm Ad}(G),
\end{equation}
where $\overline{G}$ is uniquely determined by the fact that $\overline{T} := p_1(T) \cong (\RR^\times)^{{\rm rk}(\CA)}$ is a torus and $p_2$ has finite kernel. The group $\overline{G}$ is referred to as the \emph{semisimple adjoint quotient} of $G$, and we equip it with the quotient topology with respect to the Kac--Peterson topology on $G$. We will denote by $U^{\pm}$ the positive, respectively negative maximal unipotent subgroup of $G(\Pi)$ as introduced in Section \ref{Split-real Kac--Moody groups}. Also recall from Section \ref{Split-real Kac--Moody groups} that $A_\RR:= \exp(\h_\RR(\CA)) \leq G_\CA(\RR)$ and set $A:= A_\RR\cap G$. 
\begin{lem}[(Iwasawa decomposition)]
Multiplication induces continuous bijections
\[
K_\CA(\RR) \times A_\RR \times U^+ \to G_\CA(\RR) \quad \text{and} \quad K \times A \times U^+ \to G.
\]
\end{lem}

\begin{proof}
This follows from \cite[Proposition~5.1(a)]{Kac1985}.
\end{proof}

A more refined statement has been established in \cite{FHHK} for the semisimple adjoint quotient $\overline{G}$ of $G$. To state this result, denote by 
\[G \xrightarrow{p_1} \overline{G} \xrightarrow{p_2} {\rm Ad}(G)\]
the canonical quotient maps from \eqref{AdjointSemisimpleQuotient} and set $\overline{K} := p_1(K)$, $\overline{T} := p_1(T) \cong (\R^\times)^{{\rm rk}(\mathbf A)}$, $\overline{A} := p_1(A) = \overline{T}^o$ and $\overline{U^+} := p_1(U^+)$. Equip these groups with their respective quotient topologies and note that $p_1$ restricts to a bijection between $U^+$ and $\overline{U^+}$.
\begin{thm}[{Topological Iwasawa decomposition, \cite[Theorem~3.23]{FHHK}}]\label{TopIwasawa}
Multiplication induces homeomorphisms
\[
\overline{K} \times \overline{A} \times \overline{U^+} \to \overline{G} \quad \text{and} \quad \overline{U^+} \times \overline{A} \times\overline{K}  \to \overline{G}.
\]
\end{thm}
Since $\overline{A}$ is contractible, in order to show that $\overline{K}$ is a deformation retract of $\overline{G}$ it will suffice to show that $\overline{U^+}$ is contracible. We thus need to understand the topology induced by the Kac--Peterson topology on the standard unipotent subgroups.

\subsection{The Kac--Peterson topology on \texorpdfstring{$U^{\pm}$}{Upm}}
We now turn to the study of the restriction of the Kac--Peterson topology to the standard maximal unipotent subgroups $U^-$ and $U^+$. Recall from Section \ref{Split-real Kac--Moody groups} that the Weyl group $W$ is a Coxeter group, so elements of $W$ can be represented by reduced words in the generators $s_{1}, \dots, s_{r}$. Given such a reduced word $w =( s_{i_1}, \dots, s_{i_r})$ in $W$ with corresponding simple roots $\alpha_{i_1}, \dots \alpha_{i_r}$ we define positive roots $\beta_1, \dots, \beta_r$ by
\begin{equation}\label{betas}
\beta_1 := \alpha_{i_1}, \quad \beta_2 := s_{i_1}(\alpha_{i_2}), \quad \dots, \quad \beta_r := s_{i_1}s_{i_2} \cdots s_{i_{r-1}}(\alpha_{i_r}).
\end{equation}
We then set $U_w := U_{\beta_1}\cdots U_{\beta r} \subset U^+$ and define a map
\[
\mu_w: U_{\beta_1} \times \dots \times U_{\beta_r} \to U_w, \quad (x_1, \dots, x_r) \mapsto x_1 \cdots x_r.
\]
It is established in \cite[Section 5.5, Lemma]{CapraceRemy} that the map $\mu_w$ is a bijection for every reduced word $w$, and that its image $U_w$ depends only on the Weyl group element represented by $w$, but not on the chosen reduced expression. Since $\calG_\CA(\RR)$ is a topological group, the bijection $\mu_w$ is continuous. In fact, one can show that $\mu_w$ is a homeomorphism. A proof of this fact was sketched in \cite[Lemma 7.25]{HKM}; since openness of the maps $\mu_w$ is crucial for everything that follows, we fill in the details of this sketch here:
\begin{lem} For every reduced word $w$ the map $\mu_w$ is a homeomorphism onto its image.
\end{lem}
\begin{proof} We argue by induction on the length $m$ of $w$ and observe that the case $m=1$ holds by definition. Since the linear functionals $\alpha_1, \dots, \alpha_r$ are linearly independent, there exists an element $X \in \mathfrak{h}_\RR(\CA)$ (see Section~\ref{Split-real Kac--Moody groups}) such that $\alpha_{i_1}(X) = 0$ and $\alpha_{j}(X) < 0$ for all $j \in \{1, \dots, \widehat{i_1}, \dots, r\}$. It follows that $\beta_1(X) = \alpha_{i_1}(X) = 0$ and $\beta_k(X) < 0$ for all $k =2, \dots, m$. Indeed, since the word $w$ is reduced, none of the positive real roots $\beta_2, \dots, \beta_m$ equals $\alpha_{i_1}$, and since $n\alpha_{i_1}$ is not a root for any $n \geq 2$ (cf.\ \cite[Proposition~5.1]{Kac2}), each of them contains at least one other positive simple root as a summand. Now for $j\in \{1, \dots, m\}$ and $Y \in \mathfrak g_{\beta_j}$ we have ${\rm ad}(X)(Y) = \beta_j(X)(Y)$, and thus
\[
\lim_{t \to \infty}{\rm Ad}(\exp(tX))(Y) = \left\{\begin{array}{ll}Y, & j = 1,\\ 0, & j>1. \end{array}\right.
\]
We conclude that if $x_j \in U_{\beta_j}$, then
\[
\lim_{t \to \infty} \exp(tX) (x_1 \dots x_m) \exp(-tX) = x_1,
\]
where the convergence is uniform on compacta. This shows that the map
\[
\pi_1: U_w \to U_{\beta_1}, \quad x_1 \cdots x_m \mapsto x_1
\]
is continuous, and hence the map
\begin{equation}\label{InductionStepUw}
U_w \to U_{\beta_1} \times U_{\beta_2} \cdots U_{\beta_m}, \quad x_1 \cdots x_m \mapsto (x_1, x_2\cdots x_m)
\end{equation}
is continuous. Now let $w' = (r_{i_2}, \dots, r_{i_m})$ and let $\beta_2' = s_{i_1}(\beta_2)$, \dots, $\beta_m' := s_{i_1}(\beta_m)$. Now by Axiom (RGD2) of an RGD system (see \cite[Chapter~8]{AB}) there exists an element $g \in G_\CA(\RR)$ such that $gU_{\beta_j} g^{-1} = U_{\beta_j'}$ for all $j=2, \cdots, m$, and by induction hypothesis we have a homeomorphism
\[
\mu_{w'}: U_{\beta_2'} \times \dots \times U_{\beta_m'} \to U_{w'}, \quad (x_{2}, \dots, x_m) \mapsto x_{2} \cdots x_m.
\]
Conjugating the inverse of this homeomorphism by $g^{-1}$ we obtain a homeomorphism
\[
U_{\beta_2} \cdots U_{\beta_m} \to U_{\beta_2} \times \dots \times U_{\beta_m}.
\]
Composing this homeomorphism with the map \eqref{InductionStepUw} now provides the desired continuous inverse to $\mu_w$.
\end{proof}
To describe the topology on $U^+$ we recall that there exist several distinct but related partial orders on $W$ which in different places in the literature are referred to as the \emph{Bruhat order} on $W$. In the sequel we will consider the following version; here $\ell$ denotes the length function with respect to the generating set $\{s_{1}, \dots, s_r\}$.
\begin{de} The \emph{weak right Bruhat order} on $W$ is the partial order $\leq_w$ defined as
\[
w_1 \leq_w w_2 \quad :\Longleftrightarrow \quad \ell(w_2) = \ell(w_1) + \ell(w_1^{-1}w_2).\quad (w_1, w_2 \in W)
\]
\end{de} 
According to \cite[p.~44]{CapraceRemy} we have $w_1 \leq_w w_2$ if and only if there exists a reduced word $(r_{i_1}, \dots, r_{i_{\ell(w_2)}})$ for $w_2$ such that $w_1 = r_{i_1} \cdots r_{i_{\ell(w_1)}}$. 

Recall that for the strong Bruhat order $\leq$ one has $w_1 \leq w_2$ if there exists a reduced word $(r_{i_1}, \dots, r_{i_m})$ for $w_2$ and a reduced word $(r_{j_1}, \dots, r_{j_l})$ for $w_1$ such that $(r_{j_1}, \dots, r_{j_l})$ is a substring of $(r_{i_1}, \dots, r_{i_m})$ (not necessarily consecutive). By definition,
\[
w_1 \leq_w w_2 \quad \Longrightarrow \quad w_1 \leq w_2,
\]
but the converse is not true. An important difference between the weak right Bruhat order and the strong Bruhat order is that $(W, \leq)$ contains a cofinal chain, i.e., a totally ordered subset $T\subset W$ such that for every $w \in W$ there exists $t \in T$ such that $w \leq t$, whereas for the weak right Bruhat order, such a cofinal chain does not exist. In fact, given $w_1, w_2 \in W$ there will in general not exist an element $w_3 \in W$ with $w_1 \leq_w w_3$ and $w_2 \leq_w w_3$.

Note that if $w_1 \leq_w w_2$, then we can choose a reduced word $(r_{i_1}, \dots, r_{i_{\ell(w_2)}})$ for $w_2$ such that $w_1 = r_{i_1} \cdots r_{i_{\ell(w_1)}}$. Thus if we define $\beta_1, \dots, \beta_{\ell(w_2)}$ as above then we have a commuting diagram
\[\begin{xy}\xymatrix{
U_{\beta_1} \times \dots \times U_{\beta_{\ell(w_1)}} \ar[rrr] \ar[d]&&&U_{\beta_1} \times \dots \times U_{\beta_{\ell(w_2)}}\ar[d]\\
U_{w_1} \ar[rrr] &&& U_{w_2},
}\end{xy}\]
where the horizontal maps are inclusions, and the vertical maps are homeomorphisms. In particular, we have a continuous inclusion $\iota_{w_1}^{w_2}: U_{w_1} \hookrightarrow U_{w_2}$, hence we may form the colimit
\[
\lim_{\to}((U_w)_{w \in W}, (\iota_{w_1}^{w_2})_{w_1 \leq_w w_2})
\]
in the category of topological spaces. We emphasize that in view of the previous remark the system $((U_w)_{w \in W}, (\iota_{w_1}^{w_2})_{w_1 \leq_w w_2})$ is \emph{not} directed, hence this colimit is not a direct limit.
\begin{prop} \label{directlimit}
The $k_\omega$-space $U^+$ is given by the colimit 
\[U^+ = \lim_{\to}((U_w)_{w \in W}, (\iota_{w_1}^{w_2})_{w_1 \leq_w w_2})\] 
both in the category of topological spaces and in the category of $k_\omega$-topological spaces.
\end{prop}
\begin{proof} The corresponding statement in the category of sets is established in \cite[Theorem~5.3]{CapraceRemy}. For the topological statement see \cite[Proposition~7.27]{HKM}.
\end{proof}
In view of the applications to Kac--Moody symmetric spaces that we have in mind we recall that $U^{\pm}$ are subgroups of the commutator subgroup $G$ of $G_\RR(\CA)$, in particular we can consider their images $\overline{U}^\pm := p_1(U^{\pm})$ under the map $p_1: G \to \overline{G}$ from \eqref{AdjointSemisimpleQuotient}. In this context we will need the following fact:
\begin{prop}\label{UvsUbarTopology} The map $p_1$ induces homeomorphisms $U^{\pm} \to \overline{U}^\pm$.
\end{prop}
\begin{proof} By \cite[Proposition~7.27]{HKM} the map $T \times U^+ \to TU^+$ is a homeomorphism and the kernel of $p_1$ is contained in $T$. The latter implies that $p_1$ restricts to a continuous bijection $U^+ \to \overline{U^+}$, and the former implies that this bijection is open.
\end{proof}
%If ${\bf A}$ happens to be of $2$-spherical type we can provide an alternative description of the Kac--Peterson topology on $U^+$ as follows.
%*** maybe add here also the alternative descriptions of the Kac--Peterson topology? ***
\subsection{Dilation structures on \texorpdfstring{$U^{\pm}$}{Upm}}

\begin{de} Let $U$ be a topological group. By a \emph{dilation structure} on $U$ we mean a family of maps $(\Phi_t: U \to U)_{t \in \RR}$ with the following properties:
\begin{enumerate}
%[(DS1)]
\item Each $\Phi_t$ is a continuous automorphisms of the topological group $U$.
\item $(\Phi_t)_{t \in \RR}$ is a $1$-parameter group, i.e. $\Phi_0 = {\rm Id}$ and $\Phi_{s+t} = \Phi_s \circ \Phi_t$ for all $s, t\in \RR$.
\item If we define $\Phi_{-\infty}: U \to U$ by $\Phi_{-\infty}(u) := e$, then the map
\[
[-\infty, \infty) \times U \to U, \quad (t, u) \mapsto \Phi_t(u)
\]
is continuous.
\end{enumerate}
\end{de}
\begin{rem}
\label{dilation implies contractible}
Note that if a topological group $U$ admits a dilation structure, then it is in particular contractible. Indeed, if we define $\Psi_t := \Phi_{\frac{t}{t-1}}$, then
\[
\Psi: [0, 1] \times U \to U, \quad (t,u) \mapsto \Psi_t(u)
\]
is continuous with $\Psi_0 = \Phi_0= {\rm Id}$ and $\Psi_1 = \Phi_{-\infty}$, hence a contraction to the identity.
\end{rem}
Dilation structures on finite-dimensional simply-connected nilpotent Lie groups play a major role in conducting analysis on such groups, see e.g.\ \cite{Goodman}. Not every finite-dimensional simply-connected nilpotent Lie group admits a dilation structure, but if $U$ is the unipotent radical of a minimal parabolic subgroup of a semisimple Lie group, then such a dilation structure always exists. The methods of \cite{Kum} allow one to extend this result to the Kac--Moody setting.

Following \cite[\S 3.12]{Kac2}, we define the \emph{fundamental chamber} of $ \mathfrak{h}_\RR(\mathbf{A})$ as $$C := \{ h \in \mathfrak{h}_\RR(\mathbf{A}) \mid \forall 1 \leq i \leq n: \alpha_i(h) \geq 0 \} \subset \mathfrak{h}_\RR(\mathbf{A}).$$ Since the family $(\alpha_i)_{1 \leq i \leq n}$ is linearly independent, there exists $$X_0 \in C \text{ such that } \alpha_i(X_0) = 1 \text{ for all } 1 \leq i \leq n.$$ Indeed, by the linear independence of $(\alpha_i)_{1 \leq i \leq n}$ the solution space for the system of $n-1$ linear equations $\forall 2 \leq i \leq n : \alpha_1(x) - \alpha_i(x) = 0$ has strictly larger dimension than the solution space for the system of $n$ linear equations $\forall 1 \leq i \leq n : \alpha_i(x) = 0$. 

We now define a $1$-parameter subgroup of $A_\RR$ by $a_t := \exp(tX_0)$ and denote by
\[
\phi_t := {\rm Ad}(a_t) \in {\rm Aut}(\mathfrak u^+)
\]
the associated automorphism of the Lie algebra $\mathfrak u^+ = \bigoplus_{\alpha \in \Delta_+} \g_\alpha^k$. Similarly we denote by
\[
\Phi_t := c_{a_t}|_{U^+} \in {\rm Aut}(U^+)
\]
the restriction of the conjugation-action of $a_t$ on $G_\CA(\RR)$ to $U^+$. Note that if $X \in \mathfrak u^+$ is ad-locally finite  then
\[
\Phi_t(\exp(X)) = \exp(\phi_t(X)).
\]
From \eqref{AdRootSpaces} and the defining property of $X_0$ one deduces that for every positive root $\alpha$ with height $|\alpha|$
\[
\forall Y \in \mathfrak g_\alpha: \; \phi_t(Y) = e^{t |\alpha|} Y.
\]
It follows that for all positive roots $\alpha$ one has
\begin{eqnarray}
\Phi_t(x_\alpha(s)) = \exp(tX_0) \cdot x_\alpha(s) \cdot \exp(-tX_0) = x_\alpha(e^{t |\alpha|}s),\label{contractionformula}
\end{eqnarray}
(see \cite[(4), p.~549]{Ti}), where $\{ x_\alpha(s) \mid s \in \mathbb{R}\} \cong (\mathbb{R},+)$ is the root subgroup of $G_\CA(\RR)$ corresponding to the root space $\mathfrak{g}_\alpha$.
As a consequence, if one endows each of the root subgroups $\{ x_\alpha(s) \mid s \in \mathbb{R}\}$ with the natural topology of $\mathbb{R}$, then $\Phi_t$ contracts each of them.
We are now in a position to reproduce the following result and proof by Kumar:

\begin{thm}[{\cite[Proposition~7.4.17]{Kum}}] \label{kumar}
The family $(\Phi_t)_{t\in \R}$ defines a dilation structure on $U^+$. 
\end{thm}
\begin{proof}
Let $w$ be a reduced word and write $w = s_{i_1} \cdots s_{i_r} \in W$ with corresponding simple roots $\alpha_{i_1}, \dots \alpha_{i_r}$. Recall that multiplication induces a homeomorphism
\[
U_{\beta_1} \times \dots \times U_{\beta_r} \to U_w,
\]
where the roots $\beta_1, \dots, \beta_r$ are given by 
\[
\beta_1 := \alpha_{i_1}, \quad \beta_2 := s_{i_1}(\alpha_{i_2}), \quad \dots, \quad \beta_r := s_{i_1}s_{i_2} \cdots s_{i_{r-1}}(\alpha_{i_r}).
\]
Given an element $x_{\beta_1}(y_1)x_{\beta_2}(y_2)\cdots x_{\beta_r}(y_r) \in U_w$ by \eqref{contractionformula} one has $$\Phi_t(x_{\beta_1}(y_1)x_{\beta_2}(y_2)\cdots x_{\beta_r}(y_r)) = x_{\beta_1}(e^{t |\beta_1|}y_1)x_{\beta_2}(e^{t |\beta_2|}y_2)\cdots x_{\beta_r}(e^{t |\beta_r|}y_r).$$
Setting $\Phi_{-\infty}(u):= e$ for all $u\in U^+$, we deduce that the map 
\[
\Phi|_{U_w}: [-\infty, \infty) \times U_w \to U_w, \quad (t,u) \mapsto \Phi_t(u)
\]
is continuous and that $\Phi_0 = {\rm Id}_{U_w}$. Combining this with Proposition~\ref{directlimit} one deduces that the map
\[
\Phi:  [-\infty, \infty) \times U^+ \to U^+, \quad (t,u) \mapsto \Phi_t(u)
\]
is continuous, hence a dilation structure.
\end{proof}
Recall that $U^+$ is isomorphic to $U^-$ under the Cartan--Chevalley involution of $G_\CA(\RR)$, which maps $a_t$ to $a_{-t}$. Thus if we define $\Phi^-_t := c_{a_{-t}}|_{U^-}$ then we obtain:
\begin{cor}
The family $(\Phi^-_t)_{t \in \R}$ defines a dilation structure on $U^-$.\qed
\end{cor}
Combining this with Remark \ref{dilation implies contractible} and Proposition \ref{UvsUbarTopology} we can record:
\begin{cor}\label{UContractible} The topological groups $U^+$ and $U^-$ are contractible. Consequently, the groups $\overline{U}^+$ and $\overline{U}^-$ are contractible.\qed
\end{cor}

%\subsection{The Cartan--Chevalley involution and the topological Iwasawa decomposition}

\subsection{Homotopy groups of real-split semisimple Kac--Moody groups}
\begin{cor}\label{FundamentalGroups1} The subgroup $\overline{K}< \overline{G}$ is a deformation retract. In particular the inclusion $i_K: K_\CA(\RR) \hookrightarrow G_\CA(\RR)$ is a homotopy equivalence and thus induces isomorphisms $(i_K)_*: \pi_n(\overline{K}) \to \pi_n(\overline{G})$ for all $n \geq 0$.
\begin{proof}
We have established in Corollary \ref{UContractible} that $\overline{U^+}$ is contractible, and $\overline{A}$ is contractible since it is homeomorphic to $\R^{{\rm rk}(\mathbf A)}$. The assertion now follows from Theorem \ref{TopIwasawa}.
\end{proof}
\end{cor}
Since it is currently unknown whether the Iwasawa decomposition of $G$ is also a topological decomposition, the strategy of the above proof can not be applied to $G$. However, using the following result of Palais \cite[Section 4.1, Corollary]{Pal}, one can still obtain an isomorphism between the fundamental groups of $G$ and $K$.

\begin{prop}[(Palais)]\label{Palais} Let $G$ be a topological group and let $H< G$ be a subgroup which is homeomorphic to a Lie group. Then the fibration $H \hookrightarrow G \to G/H$ is locally trivial, in particular a Hurewicz fibration, hence there is a long exact sequence of homotopy groups
\[
\dots \to \pi_2(H) \to \pi_2(G) \to \pi_2(G/H) \to \pi_1(H) \to \pi_1(G) \to \pi_1(G/H) \to \pi_0(H) \to \pi_0(G). 
\]
\end{prop}
Recall that the kernel of the quotient map $G \to \overline{G}$ is homeomorphic to $(\R^\times)^{{\rm cork}(\mathbf A)}$. In particular it has $2^{{\rm cork}(\mathbf A)}$ connected components, whereas its higher homotopy groups vanish. Applying Proposition~\ref{Palais} to the diagram of fibrations
\[\begin{xy}\xymatrix{
(\R^\times)^{{\rm cork}(\mathbf A)}\ar[r] & G \ar[r] & \overline{G} \\
(\Z/2\Z)^{{\rm cork}(\mathbf A)}\ar[u]\ar[r] & K \ar[r]\ar[u] & \overline{K} \ar[u]
}\end{xy}\]
we thus obtain:
\begin{cor} There is a commutative diagram with exact rows
\[\begin{xy}\xymatrix{
0 \ar[r] & \pi_1(G) \ar[r] & \pi_1(\overline{G}) \ar[r] & (\Z/2\Z)^{{\rm cork}(\bf A)} \ar[r] &0\\
0 \ar[r] & \pi_1(K) \ar[r] \ar[u]& \pi_1(\overline{K}) \ar[r] \ar[u]& (\Z/2\Z)^{{\rm cork}(\bf A)} \ar[u]^\cong\ar[r] &0
}\end{xy}\]
Moreover, for $n \geq 2$ there are isomorphisms $\pi_n(G_\CA(\RR)) \cong \pi_n(G)$ and $\pi_n(K_\CA(\RR)) \cong \pi_n(K)$.
\end{cor}
Combining this with Corollary \ref{FundamentalGroups1} we deduce:
\begin{thm}\label{FundamentalGroups2} For every $n \geq 0$ the inclusion $K \hookrightarrow G$ induces isomorphisms 
\[
\pi_n(K) \hookrightarrow \pi_n(G),
\]
hence is a weak homotopy equivalence. In particular, $\pi_1(G) \cong \pi_1(K)$. \hfill \qed
\end{thm}
%\begin{rem}
 %The construction of spin covers in \cite{GHKW} shows that the group $K'$ is in general not simply-connected, hence by Corollary \ref{FundamentalGroups2} the algebraically simply-connected group $G$ is not simply-connected in the topological sense either. In fact it was established very recently in \cite{Harring/Koehl}, that in the irreducible simply laced situation one always has $\pi_1(K') = \Z/2\Z$, and hence also $\pi_1(G) = \Z/2\Z$ by Corollary \ref{FundamentalGroups2}.
%\end{rem}
\subsection{Kac--Moody symmetric spaces and causal contractions}
We conclude this appendix with an application to the results obtained so far to Kac--Moody symmetric spaces. It was established in \cite{FHHK} that the homogeneous spaces $G_\CA(\RR)/K_\CA(\RR)$ and $G/K$ carry the natural structure of topological reflection spaces, and the same holds for their quotients ${\rm Ad}(G_\CA(\RR))/{\rm Ad}(K_\CA(\RR))$ and  ${\rm Ad}(G)/{\rm Ad}(K)$. The topological reflection space $\mathcal X = G/K$ is called the \emph{unreduced Kac--Moody symmetric space} of type ${\mathbf A}$, and the topological reflection space $\overline{\X} = {\rm Ad}(G)/{\rm Ad}(K) = \overline{G}/\overline{K}$ is called the \emph{reduced Kac--Moody symmetric space} of type ${\mathbf A}$. 
\begin{cor}\label{Contrac1} The reduced symmetric space $\overline{\X}$ is contractible. 
\end{cor}
\begin{proof} In view of the topological Iwasawa decomposition the orbit map at the basepoint $o = e\overline{K}$
\[
\overline{U^+} \times \overline{A} \to \overline{\X}, \quad (u,a) \mapsto ua.o
\]
is a homeomorphism. Since $\overline{U^+}$ and $\overline{A}$ are contractible, this implies contractability of $\overline{\X}$.
 %We claim that the canonical map $p: \X \to \overline{\X}$  is a Hurewicz fibration. Since its fibres are contracible, this will allow us to deduce the result for $\X$ from the result of $\overline{\X}$.
%
%To prove the claim, we consider the commutative diagram
%\[
%\begin{xy}\xymatrix{U^+ \times A
%}\end{xy}\]
\end{proof}
The proof of Theorem~\ref{kumar} can be used to provide an explicit contraction for $\overline{\X} \simeq \overline{U^+} \times \overline{A}$, using the contraction by conjugation with suitable elements of the torus $T_\RR$ on the group $\overline{U^+}$ and the standard contraction on the finite-dimensional real vector space $A$.
It turns out that this contraction has interesting additional properties. Recall from \cite[Section~7]{FHHK} that the symmetric space $\overline{\X}$ admits future and past boundaries $\Delta^+_{\|}$ and $\Delta^-_{\|}$ that both carry a simplicial structure which turns them in the geometric realizations of the positive and negative halves of the twin building of $G_\CA(\RR)$. Following \cite[Section~7]{FHHK}, a \emph{causal ray} is a geodesic ray of $\overline{\X}$ whose parallelity class equals a point in $\Delta^+_{\|}$ and a \emph{piecewise geodesic causal curve} is the concatenation of a finite set of segments of causal rays that can be parametrized in such a way that the walking direction always points towards the future boundary.
Given $x, y \in \overline{\X}$ we say that $x$ \emph{causally preceeds} $y$ (in symbols $x \preceq y$) if there exists a piecewise geodesic causal curve from $x$ to $y$.

Since both conjugation by elements of $T_\RR$ and the standard contraction of the vector space $A$ preserve geodesic rays and the future and past boundaries (cf.~\cite[Section~7]{FHHK}), the set of piecewise geodesic causal curves of $\overline{\X}$, and hence the causal pre-order $\preceq$, are invariant under the given contraction.

\begin{cor}
  The reduced symmetric space $\overline{\X}$ is causally contractible, i.e., it admits a contraction that preserves $\preceq$.
  %In particular, if $\overline{\X}$ contains a pair of $\preceq$-equivalent points, then it contains a pair of $\preceq$-equivalent points in any non-empty open subset of $\overline{\X}$. 
\end{cor}

\section{The Bruhat decomposition is a CW decomposition (by Julius Gr\"uning and Ralf K\"ohl)}

Let $G$ be a Kac--Moody group endowed with the Kac--Peterson topology and let $T$ be the standard maximal torus and $U^+$, $U^-$ the standard unipotent subgroups. \cite[Theorem~4(a)]{Kac1983} asserts without proof that the multiplication map $$U^+ \times T \times U^- \to U^+TU^-$$ is a homeomorphism with respect to the Kac--Peterson topology. In this note we provide a proof in the symmetrizable case that makes use of this fact in the two-spherical case (\cite[Proposition~7.31]{HKM}), of the embedding of Kac--Moody groups constructed in \cite[Theorem 3.15(2)]{margabberkac}, and of the fact that the Kac--Peterson topology is $k_\omega$. Among the various consequences of this result is that the Bruhat decomposition of a symmetrizable topological Kac--Moody group is a CW decomposition.   

Recall that a {\em $k$-space} (alternatively: {\em compactly generated space}) is a topological space $X$ in which a set $C \subset X$ is closed if and only if its intersection $C \cap K$ with any compact subset $K$ of $X$ is compact. That is, a $k$-space is a topological space $X$ whose topology is coherent with the family of all compact subspaces of $X$. A $k_\omega$-space is a topological space $X$ whose topology is coherent with respect to a countable ascending family of compact subspaces. By (3) of \cite{kwsurvey} any $k_\omega$-space is a $k$-space.

\begin{prop}[({\cite[Corollary]{Palais:1970}})] \label{komega}
A continuous proper map $f : X \to Y$ from a topological space $X$ to a $k$-space $Y$ is closed. In particular, a continuous injection $\iota : X \to Y$ into a $k_\omega$-space $Y=\bigcup_{n \in \mathbb{N}} Y_m$ with compact $Y_m$ such that for each $m \in \mathbb{N}$ the pre-image $\iota^{-1}(Y_m)$ is also compact is a topological embedding, i.e., it is a homeomorphism onto its image. 
\end{prop}

\begin{proof}
The first statement is exactly \cite[Corollary]{Palais:1970}. The second statement is an immediate consequence of the first, since a $k_\omega$-space is a $k$-space in which any compact subset $K$ of $Y$ is contained in some $Y_m$ of the ascending family $(Y_m)_{m \in \mathbb{N}}$ of compact subsets (statement (3) of \cite{kwsurvey}).
\end{proof}

\begin{rem}
The authors thank Tobias Hartnick and Stefan Witzel for various lively discussions concerning the correct formulation and application of Proposition~\ref{komega}. Moreover, they thank Stefan Witzel for suggesting to make use of the concept of proper maps.
%
%
%  \begin{enumerate}
%\item The authors thank Tobias Hartnick and Stefan Witzel for various lively discussions concerning the correct formulation and application of Proposition~\ref{komega}. Note in particular that the counterexample $[0, 2\pi) \stackrel{\mathrm{exp}}{\to} \mathbb{S}^1$ shows that the conclusion of the proposition becomes false if one removes the hypothesis that the intersections $Y_m \cap \iota(X)$ be compact with respect to the transported topology.
%\item Stefan Witzel pointed out to us that Proposition~1 is a variation of \cite[Corollary]{Palais:1970}: {\em A continuous proper map $f : X \to Y$ from a topological space $X$ to a $k$-space $Y$ is closed.} Here a {\em proper} map $f : X \to Y$ is a map such that pre-images of compact subsets of $Y$ are compact in $X$; and a Hausdorff space $Y$ is called a {\em $k$-space}, if a subset $A$ of $Y$ is closed if and only if any intersection $A \cap K$ with a compact subset $K$ of $Y$ is closed.
%
%  Now, a $k_\omega$-space $Y = \bigcup Y_m$ is a $k$-space in which any compact subset $K$ of $Y$ is contained in some $Y_m$ (statement (3) of \cite{kwsurvey}). Hence compactness of pre-images $\iota^{-1}(K)$ in $X$ with respect to $\tau_X$ is equivalent to compactness of intersections $K \cap \iota(X)$ with respect to $\tau_{\iota_X}$ is equivalent to compactness of intersections $Y_m \cap \iota(X)$ with respect to $\tau_{\iota_X}$ is equivalent to compactness of pre-images $\iota^{-1}(Y_m)$ in $X$ with respect to $\tau_X$.  
%\end{enumerate}
\end{rem}

A subgroup of a Kac--Moody group is {\em bounded} if it lies in the intersection of two spherical parabolic subgroups of opposite signs. In other words, it is bounded if and only if it stabilises a point the Davis CAT(0) realization of each half of its twin building. The maximal bounded subgroups of a Kac--Moody group have been determined in \cite[Theorem~4.1]{CapraceMuhlherr}.

\begin{prop}  \label{maximalbounded}
Let $G$ be a split real Kac--Moody group. Then the Kac--Peterson topology $\tau_{\mathrm{KP}}$ on $G$ equals the finest group topology $\tau_{\mathrm{MB}}$ on $G$ such that the embeddings of the maximal bounded subgroups, each endowed with its Lie group topology, are continuous. 
\end{prop}

\begin{proof}
By \cite[Lemma~4.3]{marfix}, the Kac--Peterson topology $\tau_{\mathrm{KP}}$ on $G$ induces the Lie group topology on its maximal bounded subgroups.
A fundamental $\SL_2(\mathbb{R})$ is bounded and, in particular, embeds as a closed subgroup into a maximal bounded subgroup. Therefore its subspace topology equals its Lie group topology; by \cite[Proposition~7.21]{HKM} the topology $\tau_{\mathrm{KP}}$ equals the finest group topology on $G$ such that the embeddings of the fundamental $\SL_2(\mathbb{R})$ Lie subgroups is continuous, whence $\tau_{\mathrm{KP}}$ is finer than or equal to the final group topology $\tau_{\mathrm{MB}}$ with respect to the embedded maximal bounded subgroups. Again, since by \cite[Lemma~4.3]{marfix} the Kac--Peterson topology on $G$ induces the Lie group topology on its maximal bounded subgroups, the two described topologies actually coincide.  
\end{proof}

\begin{cor} \label{maximalbounded2}
Let $G$ be a split real Kac--Moody group endowed with the Kac--Peterson topology and let $(G_i)_{i \in I}$ be a finite family of Lie-subgroups of $G$ such that each fundamental $\SL_2(\mathbb{R})$ is contained in at least one of the $G_i$. Then the Kac--Peterson topology on $G$ equals the finest group topology on $G$ such that the embeddings of the $(G_i)_i$, each endowed with its Lie group topology, are continuous. 
\end{cor}

\begin{prop}[({cf.\ \cite[1.5, 1.10]{hkl}, \cite[Theorem 3.15(2)]{margabberkac}})] \label{topemb}
Any symmetrizable topological Kac--Moody group endowed with the Kac--Peterson topology admits a continuous injective group homomorphism into a simply laced topological Kac--Moody group with closed image with respect to the Kac--Peterson topology.
\end{prop}

\begin{proof}
  By \cite[Theorem 3.15(2)]{margabberkac} for any symmetrizable Kac--Moody group $G$ there is an injective group homomorphism $\iota : G \to H$ into a simply laced Kac--Moody group $H$ embedding each fundamental rank-$1$ subgroup $G_{\alpha_i} \cong \SL_2(\mathbb{R})$ diagonally into the direct product $$\prod_{j=1}^{n_i} H_{\alpha_{i,j}} \cong \SL_2(\mathbb{R})^{n_i}$$ of a suitable (finite) family of fundamental rank-$1$ subgroups $H_{\alpha_{i,j}}$ of $H$.
  
The restriction of this map to any fundamental rank-$1$ subgroup $G_\alpha$ of $G$ is continuous with respect to the Lie group topology on $G_\alpha$ and the Kac--Peterson topology on $H$. Hence, by universality (see \cite[Proposition~7.21]{HKM}), the map $\iota : G \to H$ is continuous with respect to the Kac--Peterson topology on both $G$ and $H$. 

One has $$\iota(G)=\bigcap_\sigma \Fix(\phi_\sigma),$$
where $\phi_\sigma$ is the automorphism of $H$ given by $$H_{\alpha_{i,j}}\to H_{\alpha_{i,\sigma(j)}}$$ for some $ \sigma=(\sigma_1,\dots,\sigma_N)$, where $\sigma_i \in \mathrm{Sym}(n_i)$ acting by permuting the factors of the direct product $\prod_{j=1}^{n_i} H_{\alpha_{i,j}} \cong \SL_2(\mathbb{R})^{n_i}$. Since the automorphisms $\phi_\sigma$ are continuous with respect to the Kac--Peterson topology on $H$, the group $\iota(G)$ is a closed subgroup of $H$.
\end{proof}

%\begin{observation} \label{maxboundedcut}
%\begin{enumerate}
%\item In the situation of Proposition~\ref{topemb}, the maximal bounded subgroups of $\iota(G)$ are exactly the intersections of maximal bounded subgroups of $H$ with $\iota(G)$. Indeed, the intersection of a maximal bounded subgroup of $H$ with $\iota(G)$ is an almost connected Lie subgroup of $\iota(G)$, since $\iota(G)$ is closed in $H$, and hence is a bounded subgroup of $\iota(G)$ by \cite[Theorem~E]{marfix}. Conversely, any maximal bounded subgroup of $\iota(G)$ is an almost connected Lie subgroup of $H$, whence again is a bounded subgroup of $H$ by \cite[Theorem~E]{marfix}.   
%\item
The embedding $\iota : G \to H$ corresponds to an embedding of the twin building $\Delta_G$ of $G$ into the twin building $\Delta_H$ of $H$ such that $\Delta_G = \bigcap_\sigma \Fix(\phi_\sigma)$ (with the $\phi_\sigma$ now considered as twin building automorphisms) and the additional property that two chambers of $\Delta_G$ are opposite in $\Delta_G$ if and only if they are opposite in $\Delta_H$.

Indeed, this is immediate from an argument along the lines of descent in buildings (cf.\ \cite{MPW}). The automorphisms $\phi_\sigma$ act on the twin apartment defined by the fundamental chambers $c_+$, $c_-$ of $\Delta_H$ and, by definition, the fixed substructure is isometric to a twin apartment of $\Delta_G$. The claim then follows from the fact that $G$ acts transitively on the twin apartments of $\Delta_G$.
%\end{enumerate}

In particular, this embedding $$\Delta_G = (\Delta^+_G,\Delta^-_G, \delta^*_G) \to \Delta_H = (\Delta^+_H,\Delta^-_H, \delta^*_H)$$ of twin buildings induces an embedding of opposite geometries $$\mathrm{Opp}(\Delta_G) = \{ (c,d) \in \Delta^+_G \times \Delta^-_G \mid \delta^*_G(c,d) = 1 \} \to \mathrm{Opp}(\Delta_H) = \{ (c,d) \in \Delta^+_H \times \Delta^-_H \mid \delta^*_H(c,d) = 1 \}.$$

Specialising to the embedding of a fundamental rank-$1$ subgroup $G_{\alpha_i} \cong \SL_2(\mathbb{R})$ of $G$ diagonally into the direct product $$\prod_{j=1}^{n_i} H_{\alpha_{i,j}} \cong \SL_2(\mathbb{R})^{n_i}$$ of a suitable (finite) family of fundamental rank-$1$ subgroups $H_{\alpha_{i,j}}$ of $H$, one obtains an embedding of the real projective line $\mathbb{S}^1$ (the building of type $A_1$) diagonally into a suitable product $\left(\mathbb{S}^1\right)^{n_i}$ of real projective lines (the building of type ${A_1}^{n_i} = \underbrace{A_1 \oplus A_1 \oplus \cdots \oplus A_1}_{n_i}$).

This in turn yields an embedding of the corresponding opposites geometries of pairs of distinct points of $\mathbb{S}^1$ with adjacency relation given by the complete relation (the opposite geometry of type $A_1$ of diameter $1$), respectively of $n_i$-tuples of pairs of distinct points of $\mathbb{S}^1$ with adjacency relation given by equality in all up to at most one component (the opposite geometry of type ${A_1}^{n_i}$ of diameter $n_i$).

Refer to \cite[Section~4.3]{Gramlich} for more details, some examples, and applications of the opposite geometry. The most striking application of the opposite geometry is a proof of \cite[Theorem~13.32]{Tits1974} via its simple connectedness and M\"uhlherr's generalization to Kac--Moody groups\footnote{A manuscript that has never been published and unfortunately seems to be lost. To the second author's dismay he has lost his copy that he once owned.}; see also \cite{AbramenkoMuhlherr}.
%\end{observation}

\smallskip \noindent The following result follows immediately from the preceding discussion:

\begin{prop}\label{nondistortion}
Let $\iota : G \to H$ be the injective group homomorphism from Proposition~\ref{topemb}, let $\Delta_G = (\Delta^+_G,\Delta^-_G, \delta^*_G) \to \Delta_H = (\Delta^+_H,\Delta^-_H, \delta^*_H)$ be the induced embedding of twin buildings, and $\mathrm{Opp}(\Delta_G) = \{ (c,d) \in \Delta^+_G \times \Delta^-_G \mid \delta^*_G(c,d) = 1 \} \to \mathrm{Opp}(\Delta_H) = \{ (c,d) \in \Delta^+_H \times \Delta^-_H \mid \delta^*_H(c,d) = 1 \}$ the resulting embedding of opposite geometries. Given $(c_+,c_-) \in \mathrm{Opp}(\Delta_G)$, for all $n \in \mathbb{N}$ exists $m \in \mathbb{N}$ such that the intersection of $\mathrm{Opp}(\Delta_G)$ with the ball of radius $n$ in $\mathrm{Opp}(\Delta_H)$ around $(c_+,c_-)$ is contained in the ball of radius $m$ in $\mathrm{Opp}(\Delta_H)$ around $(c_+,c_-)$. 
\end{prop}

%\begin{proof}
%By \cite[Theorem 3.15(2)]{margabberkac}, each fundamental $G_{\alpha_i} \cong \SL_2(\mathbb{R})$ embeds diagonally into a product $\prod_{j=1}^{n_i} H_{\alpha_{i,j}} \cong \SL_2(\mathbb{R})^{n_i}$ of finitely many pairwise commuting fundamental subgroups of $H$. Since the diameter of a (twin) building and an opposites geometry of type ${A_1}^{n_i}$ is finite, in order to prove the claim of the proposition it suffices to observe that the $\mathrm{Opp}(\Delta_G)$-points of a union of bounded chains of such ${A_1}^{n_k}$-residues of $\mathrm{Opp}(\Delta_H)$ starting in $(c_+,c_-)$ actually lies in a finite $\mathrm{Opp}(\Delta_G)$-ball around $(c_+,c_-)$. However, this is obvious, because the only $\mathrm{Opp}(\Delta_G)$-points of such a ${A_1}^{n_k}$-residue come from the opposites geometry of the twin building of the diagonally embedded panel $\mathbb{P}_1(\mathbb{R}) \cong \mathbb{S}^1 \hookrightarrow \left(\mathbb{S}^1\right)^{n_i}$; in other words, each step by a ${A_1}^{n_k}$-residue in $\mathrm{Opp}(\Delta_H)$ between $\mathrm{Opp}(\Delta_G)$-points equals one step via an $\mathrm{Opp}(\Delta_G)$-panel of type $\alpha_k$.
%\end{proof}

\begin{cor} \label{multopen} Let $G$ be a topological Kac--Moody group endowed with the Kac--Peterson topology. If it is two-spherical or symmetrizable, then the multiplication map $\varphi: U^+\times T \times U^- \to G$ is a homeomorphism onto its image.
\end{cor}

\begin{proof}
The two-spherical case is \cite[Proposition~7.31]{HKM}.
In the symmetrizable case note that Proposition~\ref{komega} is applicable since the Kac--Peterson topology is $k_\omega$ by \cite[Proposition~7.10]{HKM}. Consequently, the injection from Proposition~\ref{topemb} yields a topological embedding $\iota : G \to H$, provided one can find $k_\omega$-decompositions $G = \bigcup_n G_n$ and $H = \bigcup_m H_m$ such that each intersection $H_m \cap \iota(G)$ lies in some $\iota(G_n)$. (Indeed, $\iota^{-1}(H_m)$ is closed by continuity of $\iota$, so it is compact once it lies inside some compact set $G_n$, which is equivalent to $H_m \cap \iota(G) \subset \iota(G_n)$.) 

%For $G$ choose the standard $k_\omega$-decomposition: each fundamental subgroup $G_{\alpha_{i}} \cong \mathrm{SL}_2(\mathbb{R})$ is $\sigma$-compact locally compact metrizable and, thus, admits a $k_\omega$-decomposition $G_{\alpha_{i}} = \bigcup_{n} G_{i}^n$ where each $G_i^n$ is the closed ball of radius $n$ with center $1$ in $G_{\alpha_{i}}$. By construction (see \cite[Section~7]{HKM}) $$G_1 := G_1^1, \quad G_2 := G_1^1G_2^2, \quad G_3 := G_1^1G_2^2G_3^3, \quad \cdots, G_t := G_1^1\cdots G_t^t, \quad \cdots $$ provides a $k_\omega$-decomposition of $G$, with lower indices taken modulo the number of simple roots of $G$. 

For $G$ and $H$ choose $k_\omega$-decompositions making use of Corollary~\ref{maximalbounded2} and $k_\omega$-decompositions of the fundamental subgroups $G_{\alpha_i} \cong \SL_2(\mathbb{R})$ of $G$ and the corresponding subgroups $\prod_{j=1}^{n_i} H_{\alpha_{i,j}} \cong \SL_2(\mathbb{R})^{n_i}$ of $H$ into which the $G_{\alpha_i}$ embed diagonally, endowed with their Lie group topology. That is, $$X_1 := X_1^1, \quad X_2 := X_1^1X_2^2, \quad X_3 := X_1^1X_2^2X_3^3, \quad \cdots, X_t := X_1^1\cdots X_t^t, \quad \cdots $$ where each of the $X_t^t$ is the ball of radius $t$ around $1$ of the maximal bounded subgroup $X_t$ endowed with some suitable metric inducing its Lie group topology, with $X \in \{ G, H \}$ and lower index $t$ taken modulo the total number of maximal bounded subgroups.

By construction, each $H_j^j$ intersects $\iota(G)$ in some compact subset of a fundamental subgroup $G_{\alpha_i}$ of $G$ with respect to the Lie group topology. In other words, each $H_j^j \cap \iota(G)$ lies in some $\iota(G_k^k)$.
Forming finite products of such sets and using Proposition~\ref{nondistortion} one concludes that $H_t \cap \iota(G) = (H_1^1\cdots H_t^t) \cap \iota(G)$ lies in some suitable product $G_t = G_1^1\cdots G_t^t$; that is, the injective homomorphism $\iota : G \to H$ indeed is a topological embedding.

Since $\iota$ restricts to maps $\restr{\iota}{U_+^G}:U_+^G\to U_+^H$, $\restr{\iota}{U_-^G}:U_-^G\to U_-^H$, $\restr{\iota}{T^G}:T^G\to T^H$, one can conclude that the diagram
 \[
\begin{tikzcd}
U_+^G\times T^G\times U_-^G \arrow{r}{\varphi^G} \arrow[swap]{d}{\restr{\iota}{U_+^G}\times \restr{\iota}{T^G}\times \restr{\iota}{U_-^G} }& G \arrow{d}{\iota} \\
U_+^H\times T^H\times U_-^H \arrow{r}{\varphi^H} & H
\end{tikzcd}
\]
commutes, which proves that the map $\varphi^G$ is a homeomorphism onto its image, since $\varphi^H$ is a homeomorphism onto its image by \cite[Proposition~7.31]{HKM}.
\end{proof}

\begin{cor} \label{topstrong}
 Let $G$ be a topological Kac--Moody group endowed with the Kac--Peterson topology. If it is two-spherical or symmetrizable, then the associated twin building with the quotient topology is a strong topological twin building.
\end{cor}

\begin{proof}
  The two-spherical case is \cite[Theorem~1]{HKM}. In the symmetrizable case it follows by replacing \cite[Proposition~7.31]{HKM} with Corollary~\ref{multopen}; cf.\ the discussion after \cite[Theorem~1]{HKM}.
\end{proof}

\begin{cor}
 Let $G$ be a topological Kac--Moody group endowed with the Kac--Peterson topology. If it is two-spherical or symmetrizable, then the Bruhat decomposition of a symmetrizable Kac--Moody group is a CW decomposition.
\end{cor}

\begin{proof}
This is a restatement of Proposition~\ref{bruhat decomp is cw decomp} from the main text. Its proof heavily relies on Corollary~\ref{topstrong}.
\end{proof}

\begin{cor}
Let $G$ be a topological Kac--Moody group endowed with the Kac--Peterson topology. If it is two-spherical or symmetrizable, then the coset model, the group model, and the involution model of the reduced Kac--Moody symmetric space are pairwise homeomorphic with respect to their internal topologies.
\end{cor}

\begin{proof}
  The two-spherical case is \cite[Proposition~4.19]{FHHK}. In the symmetrizable case it follows from \cite[Proposition~4.19]{FHHK} and Corollary~\ref{multopen}.
\end{proof}

\newcommand{\etalchar}[1]{$^{#1}$}
\providecommand{\bysame}{\leavevmode\hbox to3em{\hrulefill}\thinspace}
\providecommand{\MR}{\relax\ifhmode\unskip\space\fi MR }
\providecommand{\MRhref}[2]{%
  \href{http://www.ams.org/mathscinet-getitem?mr=#1}{#2}
}
\providecommand{\href}[2]{#2}

%\bibliography{bibliography}
%\bibliographystyle{amsalpha}

\end{document}